\documentclass{amsart}

\usepackage{amsmath,amsthm,amsfonts,latexsym,amssymb}
\usepackage{mathabx}
\usepackage{setspace}
\usepackage{ifthen}

\ifthenelse{\pdfoutput > 0}%
{\usepackage[pdftex]{graphicx,color}}%
{\usepackage[dvips]{graphicx,color}}

\usepackage[all]{xy}
\CompileMatrices

\DeclareMathOperator{\id}{id}
\DeclareMathOperator{\Aut}{Aut}

\DeclareMathOperator*{\fibersum}{\#}
\DeclareMathOperator*{\union}{\cup}
\DeclareMathOperator*{\disunion}{\sqcup}
\newcommand{\Z}{\mathbb{Z}}
\newcommand{\Q}{\mathbb{Q}}

\newcommand{\C}{\mathbb{C}}

\newcommand{\R}{\mathbb{R}}
\newcommand{\D}{D}
\newcommand{\T}{\mathbb{T}}
\newcommand{\Tw}{{\mathbb{T}\mathrm{w}}}

\newcommand{\into}{\hookrightarrow}

\theoremstyle{definition} \newtheorem{theorem}{Theorem}[subsection]
\theoremstyle{remark}     \newtheorem{proposition}[theorem]{Proposition}
\theoremstyle{remark}     
\theoremstyle{definition} \newtheorem{definition}[theorem]{Definition}
\theoremstyle{remark}     \newtheorem{lemma}[theorem]{Lemma}
\theoremstyle{remark}     \newtheorem{lemma*}{Lemma}
\theoremstyle{remark}     \newtheorem{construction}[theorem]{Construction}

\usepackage{url}
\usepackage{hyperref}
\hypersetup{%
  letterpaper=true,%
  bookmarks=false,%
  pdffitwindow=true,%
  pdftitle={Monopoles and Invariants of Montesinos Twins},%
  pdfauthor={Adam C. Knapp},%
  pdfnewwindow=true,%
  colorlinks=true,%
  linkcolor=blue,%
  anchorcolor=blue,%
  citecolor=blue,%
  filecolor=blue,%
  menucolor=blue,%
  pagecolor=blue,%
  urlcolor=blue,%
}

\title[Invariants of Montesinos Twins]%
{Invariants of Montesinos Twins}

\author[A. C. Knapp]{Adam C. Knapp}

\address{Department of Mathematics\\ Columbia University\\ New York, NY 10027}

\email{\href{mailto:knappa@math.msu.edu}{knappa@math.msu.edu}}


\copyrightinfo{2012}{Adam C. Knapp}

\thanks{Partially supported by NSF Grant DMS0305818 and DMS0244663 (FRG)}

\makeindex

\begin{document}
\ifthenelse{\pdfoutput > 0}%
{\DeclareGraphicsExtensions{.pdf,.jpg,.png,.mps}}%
{\DeclareGraphicsExtensions{.mps,.eps,.ps,.jpg,.png}}

\bibliographystyle{alpha}

\begin{abstract}
  In \cite{MR674227}, C. Giller proposed an invariant of ribbon
  $2$-knots in $S^4$ based on a type of skein relation for a
  projection to $\R^3$. In certain cases, this invariant is equal to
  the Alexander polynomial for the $2$-knot.  Giller's invariant is,
  however, a symmetric polynomial --- which the Alexander polynomial
  of a $2$-knot need not be.  After modifying a $2$-knot into a
  Montesinos twin in a natural way, we show that Giller's invariant is
  related to the Seiberg-Witten invariant of the exterior of the twin,
  glued to the complement of a fiber in $E(2)$.
\end{abstract}

\maketitle

\section{Introduction}

A smooth {\em $2$-knot} \index{knot!$2$-knot} $K$ in the $4$-sphere
is an embedded $S^2$ in $S^4$ considered up to smooth isotopy.  We
say that $K$ is unknotted when it bounds a smoothly embedded
$\D^3\subset S^4$.  In contrast to classical knot theory, $2$-knots
display many pathologies; their exteriors need not be a $K(\pi_1,1)$
\cite{MR0107239} nor does the homeomorphism type of the knot
complement determine the knot's isotopy class
\cite{MR0440561}. % (Does diffeomorphism type?)
After a surgery which replaces the regular neighborhood $\nu K \cong
S^2\times \D^2$ with $\D^3\times S^1$ we obtain a homology $S^1\times
S^3$.  With $b_2^+=b_2=0$ for such a manifold, current tools for
smooth $4$-manifolds offer few invariants.

However, the natural generalizations of gadgets like the Alexander
polynomial or ideal can be computed for $2$-knots.  In
\cite{MR674227}, C.  Giller proposed a definition for an invariant we
will call $\Delta_G$.  The proposed invariant is derived from a kind
of skein relation on the projection to $\R^3$ of a $2$-knot analogous
to the Conway skein relation for projections of $1$-knots to
$\R^2$. For a certain class of $2$-knots, $\Delta_G$ is known to
compute the Alexander polynomial.  This relation, which we discuss
more thoroughly in section~\ref{sec:giller}, always results in a
symmetric polynomial.  However, the Alexander polynomial of a $2$-knot
need not be symmetrizable.  For example, the knot shown in
Figure~\ref{fig:giller-ex} has Alexander polynomial $\Delta_A=1-2t$
while Giller's polynomial is $\Delta_G=t^{-2}-1+t^{2}$. Further
complicating matters is that Giller's polynomial was not actually
known to be an invariant.

If instead of $2$-knots, we look at ``twins'' in $S^4$, we have
objects to which gauge-theoretic methods can be applied.  It is in
this context we see that Giller's polynomial computes, in the relevant
cases, the Seiberg-Witten invariant of the exterior of the twin.

\subsection{Twins} 

In \cite{MR698205} and \cite{MR734666}, J. M.  Montesinos introduced
the concept of a twin in $S^4$.  Such an object consists of a pair of
$2$-knots $K_1,K_2$ which meet transversely
twice.\index{twin!Montesinos} As the second homology of $S^4$ is
trivial, these intersection points must come with opposite signs.  By
``standard twins'', we will mean that both $K_i$ are unknotted and
that the exterior of the pair is diffeomorphic to $T^2\times
\D^2$. \index{twin!standard} We construct these explicitly below. In
general, the exterior of a twin is a homology $T^2\times \D^2$ with
boundary diffeomorphic to $T^3$.

As $\pi_2(SO(2))$ is trivial, an orientation on an $2$-knot $K$ in
$S^4$ determines a trivialization $\nu S^2 \cong S^2\times \D^2$.
That is, for a $2$-knot $K$, all framings are equivalent to the
Seifert\footnote{A Seifert manifold for a surface $\Sigma$ in $S^4$ is
  a $3$ manifold $M$ with boundary diffeomorphic $\Sigma$, smoothly
  embedded in $S^4$ so that $\partial M = \Sigma$.  As in
  \cite{MR0515288}, Seifert manifolds exist because of the following:
  Consider the map $\Sigma\times \partial \D^2 \to \partial \D^2$
  which is given by a trivialization of the normal bundle of
  $\Sigma$. Obstructions to extending this map over all of
  $S^4\setminus \nu \Sigma$ vanish, giving us a map $S^4\setminus \nu
  \Sigma \to S^1$. We can homotope this map to a smooth map which
  remains equal to the projection $\Sigma\times \partial \D^2
  \to \partial \D^2$ on a tubular neighborhood of the boundary. Then
  as $S^4\setminus \nu \Sigma$ is compact there are a finite number of
  critical points. Let $M$ be the preimage of one of the regular
  points. } framing.  Fix orientations on the $2$-knots, $K_1,K_2$, in
a twin $\Tw$ and consider $\partial \nu \Tw \cong T^3$.  Take a simple
closed curve $\gamma$ on the twin passing through the intersection
points of $K_1,K_2$ and lift it to $\partial \nu \Tw$ using the
framings.  All such $\gamma$ are isotopic on the twin and, as each
$K_i$ has a single framing, this lift is canonical.  This decomposes
$\partial \nu \Tw = \gamma \times T^2$ The boundaries of normal discs
$D^2_1,D^2_2$ to $K_1,K_2$ refine the decomposition to $\partial \nu
\Tw = \gamma \times \partial D^2_1 \times \partial D^2_2$.

With this decomposition in mind, we define a standard surgery for a
twin $\Tw$ in $S^4$. Let
\begin{displaymath}
  S^4_{\Tw}
  =
  (S^4\setminus \Tw) \union_{\phi} (T^2\times \D^2)  
\end{displaymath}
where $\phi:T^3 \to T^3$ identifies the $\partial D^2$ from $T^2\times
\D^2$ with $\gamma$.  Up to isotopy, any two such $\phi$ differ by a
diffeomorphism $T^2\times\text{pt}\to T^2\times\text{pt}$. Since any
such map extends over $T^2\times \D^2$, the resulting manifold only
depends on the embedding of the twin.  Also, $H_*(S^4_{\Tw};\Z)\cong
H_*(T^2\times S^2;\Z)$ so that the $2$-torus core, $T_{\Tw}$, of the
surgery is identified with $T^2\times pt\subset T^2\times S^2$ as a
homology class.

\subsection{Definition of the invariant for twins}

Let $F$ be a smooth fiber in an elliptic fibration of $E(2)$, the
underlying smooth manifold for a complex $K3$ surface.  The regular
neighborhood of the fiber, $\nu F$, comes with a trivialization $\nu
F=D^2\times T^2$ induced by the fibration giving us the identification
$\partial \nu F = \partial D^2\times F$.  Taking a twin $\Tw \subset
S^4$, we also have a decomposition $\partial{N}(\Tw)=\gamma\times
\partial D^2_1 \times \partial D^2_2$.  If we fix an identification of
this $\partial D^2_1 \times \partial D^2_2$ and $F$
and an orientation reversing diffeomorphism between $\partial D^2$ and
$\gamma$, we obtain an identification $\phi$ of
$\partial(E(2)\setminus \nu F)$ and $\partial(S^4\setminus \Tw)$.
Let
\begin{displaymath}
  E(2)_{\Tw}
  =
  (E(2)\setminus \nu F)\union_\phi (S^4\setminus \Tw).  
\end{displaymath}
Since any automorphism of $F$ exends smoothly over $E(2)\setminus \nu
F$, this construction depends on the smooth isotopy class of $F$ but
is independant, up to diffeomorphism, of the particular choice of
$\phi$. As $\pi_1(E(2)\setminus \nu F)=0$ and the image of
$\pi_1(T^3)$ normally generates $\pi_1(S^4\setminus \Tw)$,
$E(2)_{\Tw}$ is simply connected. Then an easy Meyer-Vietoris argument
shows that $E(2)$ and $E(2)_{\Tw}$ have the same cohomology ring. By
Freedman's theorem, $E(2)$ and $E(2)_{\Tw}$ are homeomorphic.

This procedure may also be thought of as the generalized fiber sum of
$E(2)$ and $S^4_{\Tw}$ along $F$ and the core $T_{\Tw}$ of the
surgered twin with its default framing.  As such, we will sometimes
write $E(2)_{\Tw}$ as $E(2) \fibersum_{F=T_{\Tw}} S^4_{\Tw}$. 

The invariant of the twin $\Tw$ which we will consider will be the
Seiberg-Witten invariant of $E(2)_{\Tw}$.  i.e.
\begin{definition}
  $I(\Tw) = SW(E(2)_{\Tw})$ thought of as an element of the group ring
  $\Z\left[ H_{2}(E(2))\right]$.
\end{definition} \index{twin!invariant} 

Precisely, $SW(E(2)_{\Tw})$ lives in $\Z\left[ H_2 (E(2)_{\Tw})
\right]$ which we identify with $H_{2}(E(2))$ by a homomorphism which
extends the identity map on $H_{2}(E(2)_{\Tw}\setminus (S^4\setminus
\nu \Tw))\to H_2(E(2)\setminus \nu F)$. The invariant is well defined
up to a sign which depends on a homology orientation of
$H_0(E(2))\oplus \det H_{2}^{+}(E(2)) \oplus H_{1}(E(2))$

As the standard twins $\Tw_{std}$ have complement diffeomorphic to
$T^2\times \D^2$, $E(2)_{\Tw_{std}}$ is diffeomorphic to $E(2)$ and
\begin{equation}
  \label{eq:I-std-twin-relation}
  I(\Tw_{std}) = SW(E(2)) = 1.  
\end{equation}

Now suppose that we have a twin $\Tw$ and a disjoint torus $T$ in
$S^4$.  As $T$ is null-homologous, it has a canonical framing given by
the outward normal of its Seifert manifold. We can then form the self
fiber sum $S^4_{\Tw}\fibersum_{T_\Tw=T}$. With a choice of an
identification of $T_\Tw$ and $T$, the result is well defined and any
choice of identification will have the same Seiberg-Witten polynomial,
due to the gluing formulas of \cite{MR1492130}. Then form
\begin{displaymath}
 E(2)_{\Tw,T}
 =
 E(2) \fibersum_{F=T_{\Tw}'} \left( S^4_{\Tw}\fibersum_{T_\Tw=T} \right) 
\end{displaymath}
with $T_{\Tw}'$ a pushoff of $T_{\Tw}$.  Our new manifold
$E(2)_{\Tw,T}$ is a homology $E(2)\fibersum_{F=F'}$ with $F,F'$
elliptic fibers and $SW(E(2)\fibersum_{F=F'})=(t-t^{-1})^2$ with
$t=\exp([F])$. See \cite{MR1492130}. We define
\begin{definition}
  $I(\Tw,T) = (t-t^{-1})^{-1}SW(E(2)_{\Tw,T})$ thought of as a
  element of the group ring $\Z\left[H_2(E(2);\Z)\right]$. 
\end{definition} \index{twin!twin-torus invariant} 

\subsection{Construction of Knotted twins}

\begin{construction}[Connect sum]
  Take $K_0$, a knotted $S^2$, and a twin $\Tw=K_1\union K_2$ in
  $S^4$.  Then, selecting one of the spheres $K_1$ in the twin, we form
  the connected sum $(S^4\fibersum S^4,K_0\fibersum K_1)$ at some
  point away from the 2 double points of the twins.  (This construction
  is not independent of the choice of $K_1,K_2$ in $\Tw$ in general.)
\end{construction} \index{connect sum}

If we take $\Tw$ to be the standard twins, then this construction {\em
  is} independent of the choice of $S^2$s and provides a handy method
for studying $2$-knots via twins.  This independence is due to the
existence of a orientation preserving diffeomorphism $\rho$ of $S^4$
which interchanges $K_1,K_2$ in standard twins. The diffeomorphism
$\rho$ is constructed as follows: View $S^4$ as $\nu \Tw_{std}\union
T^2\times \D^2$. Define $\rho$ on $\nu \Tw_{std}$ to be the obvious
map which interchanges $K_1$ and $K_2$. Then $\rho$ induces the map
\begin{displaymath}
  \left[\begin{array}{ccc}
      0 & 1 & 0 \\
      1 & 0 & 0 \\
      0 & 0 & 1 
  \end{array}\right]
\end{displaymath}
on $\partial \nu \Tw_{std} = T^2\times \partial \D^2 = T^3$ under the
basis for $H_1(T^3)$ given by $\partial D^2_{K_1},\partial
\D^2_{K_2},$ and $\partial \D^2_{T^2}$ where $\partial D^2_{S}$ is the
boundary of the normal bundle to $S$. This map then extends over
$T^2\times \D^2$ giving $\rho$ on $S^3$.

\begin{construction}[Artin Spin]
  This construction is originally due to E. Artin in \cite{Artin1925}.
  To each smooth knot $K$ in $S^3$ and let $\widehat{K}$ be the
  corresponding arc in $\D^3$ obtained by removing a small $\D^3$
  surrounding a point of $K$.  We can arrange $\widehat{K}$ so that it
  meets $\partial \D^3$ perpendicularly at the north and south
  poles. Clearly $S^3\setminus \nu K$ and $\D^3\setminus \nu
  \widehat{K}$ are diffeomorphic.

  Using the decomposition $S^4 = \D^2\times S^2 \cup S^1 \times \D^3$,
  we can identify the $4$-sphere with the space $(S^1 \times
  \D^3)/\sim$ where $(\theta_0,x_0)\sim (\theta_1,x_1)$ iff
  $x_0=x_1\in \partial \D^3$.

  Let $K_1, K_2$ be, respectively, the images of $\{1\}\times\partial
  \D^3$ and the $S^1\times \dot{K}$ under $\sim$. Each is a smoothly
  embedded $2$-sphere which intersect each other pairwise at the
  images of the north and south poles of $\partial \D^3$. Thus $K_1$
  and $K_2$ form twins $\Tw_K$. The first sphere, $K_1$, bounds
  $\{1\}\times \D^3$ in $S^4$ and if $\Sigma$ is a Seifert surface for
  $K$, then $K_2$ bounds $S^1\times\Sigma$ union a $2$-handle attached
  along $S^1\times \text{pt}$. The loop $\gamma$, connecting the
  intersection points of the $K_i$, bounds a copy $\Sigma$ in the
  complement of the twin.

  In fact, $S^4 \setminus \nu \Tw_K$ is diffeomorphic to $S^1 \times
  (S^3 \setminus \nu K)$. In the case that $K$ is the unknot, we get
  the standard twins $S^4 \setminus \nu \Tw_K = S^1\times(S^1\times
  \D^2) = T^2\times \D^2$.  See Figure~\ref{fig:stdtwins}.
  
\end{construction} \index{spin!Artin}

\begin{figure}[htbp]
  \centering
  \includegraphics{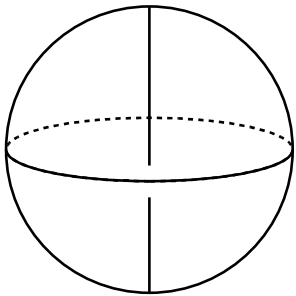}
  \caption[Standard Twins]{Spin to get standard twins}
  \label{fig:stdtwins}
\end{figure}

When the standard twin surgery is performed on twins formed by Artin
spinning $K$, the result is the manifold $S^1\times S^3_0(K)$.  (Where
$S^3_0(K)$ is the result of zero surgery on $K\subset S^3$.) This case
is identical to the knot surgery considered by Fintushel and Stern in
\cite{MR1650308} and so we have
\begin{theorem}[Fintushel, Stern]
  $I(\Tw_K) = \Delta_K(t^2)$ where $\Delta_K$ is the symmetrized
  Alexander polynomial of $K$ and $t=\exp([F])$.
\end{theorem}

\begin{construction}[Twist Spin]
  As before, given a knot $K$ in $S^3$ we obtain a knotted arc
  $\widehat{K}$ in $\D^3$ with boundary on the north and south poles.  For
  each $\theta\in S^1=\R/2\pi\Z$ let $\widehat{K}_{n\theta}$ be the image
  of $\widehat{K}$ rotated by $n\theta$ radians about the $z$-axis. The
  annulus $A_n$ in $S^1 \times \D^3$ is obtained by taking the union
  of $\{\theta\}\times \widehat{K}_{n\theta}$. This annulus descends to a
  knotted $S^2$, $K_2$, in $S^4=S^1\times B^3/\sim$.

  Together with $K_1$, the image of $\{\theta\}\times \partial B^3$,
  we form a twin.  We call $K_2$ the $n$-twist spin of $K$ and write
  $K_2=\tau^n K$ and $\Tw_{\tau^n K}$ for the associated twin.

  As with the Artin spin, the first sphere, $K_1$, bounds $\{1\}\times
  \D^3$ in $S^4$. The loop $\gamma$, connecting the intersection
  points of the $K_i$, bounds a copy $\Sigma$ in the complement of the
  twin.

\end{construction} \index{spin!twist} 

This construction is due to Zeeman, who in \cite{MR0195085}, showed
that the $n$-twist spin of a knot was not isotopic to any $2$-knot
obtained by Artin spinning when $n>1$.  The $n=1$ case is interesting
in that the $1$-twist spin, $\tau K$ is unknotted independently of
choice of $K$. However, the twin $\Tw_{\tau K}$ is typically
knotted. This provides an interesting example of twins in which both
$2$-knots are unknots but with the twins being knotted as an
interchangable pair. See \cite{MR487047}.

\begin{construction}[Roll Spin]
  Similar to the twist spin, this construction involves a deformation
  of a knotted arc $\widehat{K}$ in $\D^3$ fixing the north and south
  poles which returns the arc to the starting point.  Take a
  international date line of $\partial \D^3$ union $\widehat{K}$ and push
  it into $\D^3\setminus\widehat{K}$ so that it is null homologous.  Call
  this $\widehat{K}$.  Then consider the $1$-parameter family of
  diffeomorphisms given by pushing a base point $x$ $n$ times along
  $\widehat{K}$.  The return map $\phi$ is then a diffeomorphism of the
  quadruple $(\D^3,\partial \D^3,\widehat{K},x)$ which is the identity on
  all but the first component. Since $\phi$ is isotopic to the
  identity rel $\partial \D^3$, $S^1 \times \D^3$ is diffeomorphic to
  $\R \times D^3/ (+1,\phi)$. Let the annulus $A$ be the image of $\R
  \times \widehat{K}/(+1,\phi)$ in $S^1\times D^3$. 

  Then, as before, $A$ becomes a $2$-knot $K_2$ in $S^4$ after
  quotienting by $\sim$.  We call $K_2$ the $n$-roll spin of $K$ and
  write $K_2=\rho^n K$ and $\Tw_{\rho^n K}$ for the associated twin.
\end{construction} \index{spin!roll}

The roll spinning construction is due to Fox who, in \cite{MR0184221},
showed that, for $\widehat{K}=\widehat{4_1}$, the knotted $2$-sphere $K_2$
coming from the deformed arc is not isotopic to any $n$-twist spun
knot.  In this case, the $1$-roll spin had a corresponding
visualization, duplicated in Figure~\ref{fig:foxroll}, of the motion
of $\widehat{K_1}$ in $\D^3$ which explains why the word ``roll'' was
chosen to describe this construction.

\begin{figure}[phtb]
  \centering
  \includegraphics{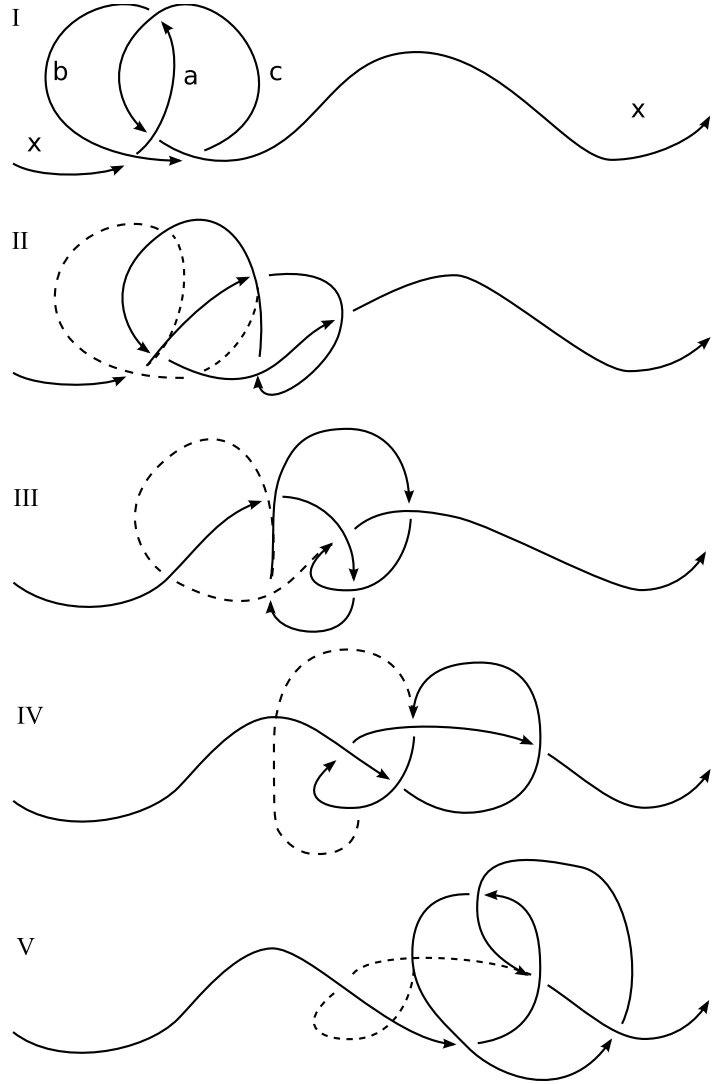}
  \caption{Fox's Roll Spin}
  \label{fig:foxroll}
\end{figure}

Note that both twist and roll spinning can be described in terms of
certain diffeomorphisms of $\D^3$ which keep $\partial \D^3$ and
$\widehat{K}$ fixed identically.  With this in mind, we now consider their
mutual generalization, Deform spinning:
\begin{construction}[Deform Spin]
  Let $\phi$ be a self diffeomorphism of $\D^3$ keeping $\partial
  \D^3, \widehat{K}$ and a base point $p$ fixed identically.  Then the
  mapping torus of $\phi$ is $S^1\times \D^3\cong (\R\times
  \D^3)/\left<(r,x)\sim(r+1,\phi(x))\right>$ with an embedded annulus
  $\widehat{K_2}$ which is the quotient of $\R\times \widehat{K}$.
  Then after quotienting by the identification $(\theta_0,x)\sim
  (\theta_1,x)$ for $x \in \partial \D^3$ as before, we obtain a
  knotted $2$-sphere $K_2$, the image of $\widehat{K_2}$, which
  together with $K_1$, the image of $\{\theta\}\times \partial \D^3$,
  forms a twin.  We write $K_2=\phi K$ and $\Tw_{\phi K}$ for the twin
  pair.  The isotopy class of $K_2$ and of $\Tw_{\phi K}$ is
  determined solely by the isotopy class of $\phi$.
\end{construction} \index{spin!deform}

This construction was introduced by Litherland in \cite{MR530058}.
Diffeomorphisms such as $\phi$ are called ``deformations'' and form a
group $\mathcal{D}(K)$ of deformations modulo isotopy.
$\mathcal{D}(K)$ is isomorphic to the group
$\Aut_{\partial}(\pi_{1}(S^3\setminus K))$ of automorphisms of
$\pi_{1}(S^3\setminus K)$ preserving the fixed peripheral subgroup
given by the image $\pi_1(\partial \nu K) \to \pi_1(S^3\setminus \nu
K)$.  In this setup we find that $\tau$ corresponds to conjugation by
the meridian of $K_1$ and $\rho$ corresponds to conjugation by the
longitude of $K_1$.  Then
\begin{lemma}[Litherland]
  If $K_1$
  \begin{itemize}
  \item is not a torus knot, $\Z_{\tau}\oplus \Z_{\rho} \subset
    \mathcal{D}(K_1)$.
  \item is a $p,q$-torus knot, $\tau^{pq}\rho=\id$ and
    $\mathcal{D}(K_1)\cong\Z_{\tau}$.
  \end{itemize}
\end{lemma}
In the case that $K$ is a composite knot, $K= \#_{i \in I} K_i$, the
deformation group $D(K)$ may be much larger than the subgroup
generated by $\rho,\tau$ and includes a copy of the pure braid group
on $|I|$ strands; see \cite{MR0461567}. % \cite{MR2679699}

\subsection{Ribbon Knots and Twins}\label{sec:ribbon}

We say that a 2-knot $K$ is ribbon if it is formed by the following
construction: \index{ribbon!knot} Let $D=\disunion \D^3$ (bases),
$B=\disunion \D^2\times I$ (bands) each be embedded in $\R^4$ with
$(\partial D)\cap B = \disunion \D^2\times \pm1$. If a band intersects a
base elsewhere, 
\begin{displaymath}
  \left(\D^2\times (-1,1)\right)\cap \D^3 = \D^2\times t, \quad t\in (-1,1)  
\end{displaymath}
and
\begin{displaymath}
  \left(\D^2\times(-1,1)\right)\cap (\partial \D^3)=\emptyset. 
\end{displaymath}
The second type of intersection is called a ribbon intersection of $K$
or ribbon singularity of $D\union B$.  Then
\begin{displaymath}
  K=(\partial D\setminus 
  \underbrace{\disunion \D^2 \times \pm1}_{\text{in } \partial B}) 
  \union
  (\disunion (\partial \D^2)\times I)    
\end{displaymath}
is a ribbon knot with ribbon presentation given by $D\union B$.  We can
define ribbon surfaces of arbitrary genus in the same manner.

Suppose that we have a twin $\Tw=K_1\union K_2$ for which the $K_i$ are
ribbon.  Then for each $K_i$ we have a set of bases $D_i$ and bands
$B_i$.  We will say that $\Tw$ is ribbon if
\begin{enumerate}
\item $B_1\cap B_2 = \emptyset$

\item $D_1\cap B_2 = \disunion \D^2 \times t$, $t\in(-1,1)$ with
  $(\partial D_1) \cap D^2\times (-1,1) = \emptyset$ for each band in
  $B_2$ (ribbon intersection)

\item $D_2\cap B_1 = \disunion \D^2 \times t$, $t\in(-1,1)$ with
  $(\partial D_2) \cap D^2\times (-1,1) = \emptyset$ for each band in
  $B_1$ (ribbon intersection)

\item $D_1 \cap D_2 = 2\D^2$.  More specifically, $D_1$ and $D_2$ meet
  only in two balls $D'_i,D''_i$ from each and at these intersections
  (corresponding to the intersection points $K_1\cap K_2$), we have
  the following local model:
  \begin{displaymath}
    D'_1=\left\{ (z_1,z_2)\in \C^2 \mid \text{Re}( z_2 ) \leq 1,
      \text{Im}( z_2 ) =0 \right\}\subset \C^2,
  \end{displaymath}
  \begin{displaymath}
    D'_2=\left\{ (z_1,z_2)\in \C^2 \mid \text{Re}( z_1 ) \leq 1,
      \text{Im}( z_1 ) =0 \right\}\subset \C^2,
  \end{displaymath}
  So that taking the boundary of each gives us the cone of the
  positive Hopf link.  The $D''_1,D''_2$ case is the same but with
  orientations reversed on $D''_2$, giving us a negative Hopf link
  cone boundary.

\end{enumerate}
In this case we say that $\Tw$ is a \emph{ribbon
  twin}.  \index{ribbon!twin} \index{twin!ribbon}

It is known that of the deform spun knots, only the Artin and 1-twist
spun knots are ribbon. Hence, at most Artin and 1-twist spun twins may
be ribbon. Artin spun twins are known to be ribbon but it is not known
to the author if 1-twist spun twins are ribbon.

Two ribbon presentations are {\em stably equivalent} if they are
equivalent under the following operations. Addition of a trivial
base/band pair, sliding the disc to which a band attaches (band
slide), and moving a ribbon intersection along a base/band sequence
(band pass) are shown in Figure~\ref{fig:band-stabilization} and
together with isotopy generate {\em stable equivalence } of ribbon
presentations.  Clearly stable equivalence of ribbon presentations
generates isotopies of the corresponding ribbon knot but the converse
also holds --- isotopic ribbon knots have stably equivalent ribbon
presentations. For a proof of this, see \cite{MR1180400}.

\begin{figure}[htbp]
  \centering
  \includegraphics{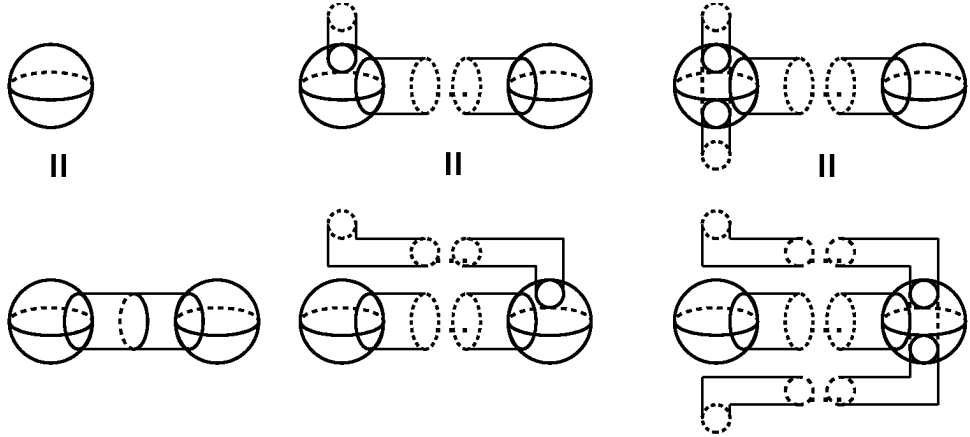}
  \caption[Trivial addition/deletion, Band slide, and Band pass]{A)
    Trivial addition/deletion B) Band slide C) Band pass}
  \label{fig:band-stabilization}
\end{figure}

It will occasionally be easier to deal with simplified ribbon
presentations.  Let $\Gamma=\Gamma(D,B)$ be the graph which has
vertices corresponding to bases and edges given by bands, connected in
the natural way.  It is clear that $b_1(\Gamma)$ is the genus of the
ribbon surface specified by the ribbon presentation.  We restrict
ourselves to the case where the ribbon surface is a sphere or torus.
Suppose that $\Gamma$ has a vertex $x$ of valence $3$ or greater.
Then one of the outgoing edges of $x$ has a path which ends at a
vertex $y$ with a single incoming edge.  Perform the band slide
corresponding to this path to get a new ribbon presentation $\Gamma'$
with the same set of bases.  In $\Gamma'$, the valence of $x$ has
decreased by $1$ and the valence of $y$ is now $2$.  Continue this
procedure until we arrive at a graph $\widehat{\Gamma}$ (and
corresponding ribbon presentation) for which each vertex has valence
at most $2$.  Then, as cell complexes, $\widehat{\Gamma}$ is either an
interval or a circle as the ribbon surface is a sphere or a torus.  We
will call such ribbon presentations {\em linear}.\index{ribbon
  presentation!linear}

Consider the connect sum of a ribbon $2$-knot $K_0$ with standard
twins $\Tw=K_1\union K_2$.  Standard twins have a simple ribbon
presentation given by two bases and a band each.  (Each base is for
one of the twin intersection points.) Stabilize the band in $K_1$ by
switching it for two bands and a base.  Then the connect sum
$K_0\fibersum K_1$ is formed by adding a band from the new base of
$K_0$ to any of the terminal bases in a ribbon presentation of
$K_1$. \index{ribbon presentation!of connected sum}

\subsection{Projections} \label{subsec:projection}

In the study of knots in $\R^3$, a generic projections to $\R^2$,
together with crossing information, completely determines the isotopy
type of a knot.  Similarly, there is a theory for decorated
projections for twins and surfaces in $\R^4=S^4\setminus \text{pt}$
which determines their embedding up to isotopy.

Giller proves, in \cite{MR674227}, that if $\Sigma$ is a surface in
$\R^4$, then up to isotopy $\Sigma$ admits a projection to $\R^3$ with
only double and triple points exist. Further, such projections are
generic.  In these generic projections, the double points either exist
in families which are either simple closed curves or embedded open
intervals whose closed endpoints are triple points.  See
Figure~\ref{fig:dbltrpl-pt}.  In the same paper, he gives methods of
decorating these projections with over -- middle -- under crossing
information and a way of determining if an arbitrary set of crossing
information gives a lift of such an immersion of a surface in $\R^3$
to an embedding in $\R^4$.

\begin{figure}[htbp]
  \centering
  \includegraphics{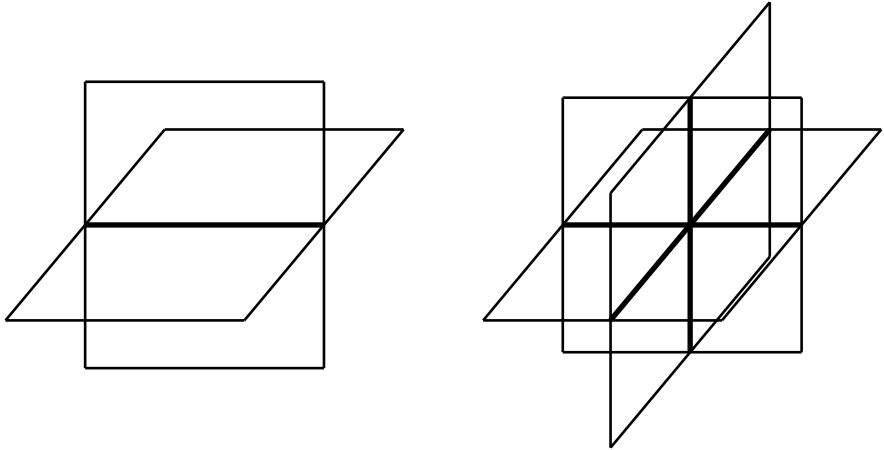}
  \caption{Local models for a family of double points and a triple point}
  \label{fig:dbltrpl-pt}
\end{figure}

We will only consider those knots and twins which admit a projection
which contain no triple points.  Not all twins or surfaces have such a
projection and those that do are said to be {\em simply knotted}.
First examples of simply knotted $2$-knots include Artin spun knots
and ribbon $2$-knots.  For Artin spun knots,
%\footnote{EXPAND}
we can get a projection with no triple points by doing the same
spinning construction (one dimension down) to the projection to $\R^2$
of the original, classical knot.  This creates an $S^1$s worth of
double points for each crossing in the classical knot's projection.
We will call a twin $\Tw=K_1\union K_2$ simply knotted if both $K_i$
are simply knotted and pairwise have no triple points.

Ribbon knots have embedded projections away from the ribbon
singularities --- the intersections of the interiors of bands with
interiors of the bases.  (This is in contrast with ribbon $1$-knots,
for which projections of non-intersecting bands may have crossings.
The analogous situation for $2$-knots is an under/over crossing of the
whole band --- which does not result in a crossing in the projection
to $\R^3$.)  Nearby the ribbon singularities, we have projections
which appear as in Figure~\ref{fig:ribbon-sing}.  It was proved in
\cite{MR0172280} that all simply knotted $S^2$s are ribbon.

\begin{figure}[htbp]
  \centering
  \includegraphics{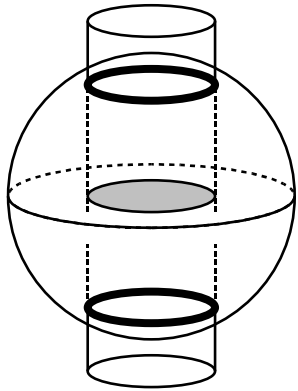}
  \caption[Projection of a ribbon singularity and corresponding double
  points.]{A projection of the neighborhood of a ribbon singularity
    (shaded disc) and the corresponding two $S^1$s of double points.}
  \label{fig:ribbon-sing}
\end{figure}

When we have a $S^1$ family of double points, we have local
neighborhoods around each which appear as in the first picture in
Figure~\ref{fig:dbltrpl-pt}.  This gives the neighborhood of the
family the structure of an bundle over $S^1$.  As the surfaces in
$\R^4$ are orientable and the two preimages of the double points are
separated, the monodromy must be trivial.  This means that, local to
the $S^1$ family of double points, the projection is that of a
classical knot crossing times $S^1$.

Then, for a simply knotted projection of a (oriented) surface in
$\R^4$, it is sufficient to label one of the surfaces as being
over crossing at each family of double points.  We will use ``+'' to
denote this.

\begin{figure}[htbp]
  \centering
  \includegraphics{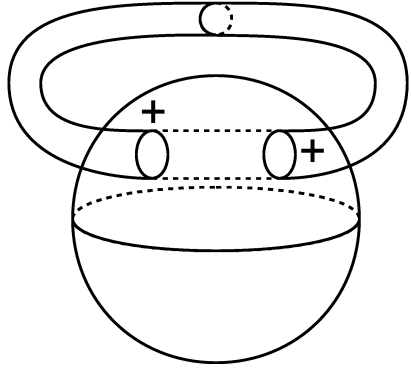}
  \caption{A sphere and a torus with crossing information.}
  \label{fig:crossing-info}
\end{figure}

\begin{figure}[htbp]
  \centering
  \includegraphics{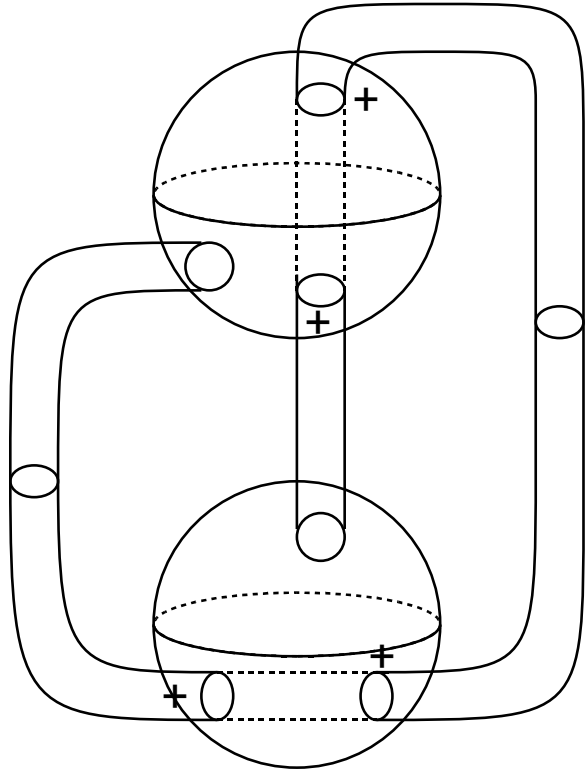}  
  \caption[A $2$-knot with $\Delta_A=1-2t$ and $\Delta_G=t^{-2} -1 +
  t^2$]{A $2$-knot with Alexander polynomial $1-2t$ and Giller
    polynomial $\Delta_G=t^{-2} -1 + t^2$}
  \label{fig:giller-ex}
\end{figure}

For a twin, a few additional pieces of information are needed.  We
need to keep track of the two intersection points of the spheres.  In
$S^4$, the neighborhood of each is diffeomorphic to the cone on a
positive or negative Hopf link.  Then the (undecorated) projection of
such a neighborhood appears as does a neighborhood of double points.
We decorate the projection with a solid dot to indicate the
intersection point of the $S^2$s and $+$ signs to indicate over/under
crossings on the double point arcs which emanate from it.  We switch
from over to under at the intersection of the spheres in twins.  See
Figure~\ref{fig:twin-int}

\begin{figure}[htbp]
  \centering
  \includegraphics{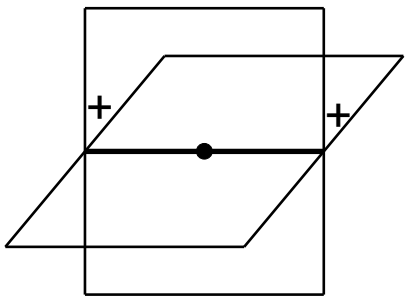}
  \caption[Neighborhood of the intersection of spheres in a twin]{The
    projection of the neighborhood of the intersection of spheres in a
    twin.  To the left, the vertical plane crosses above and to the
    right, the horizontal does.}
  \label{fig:twin-int}
\end{figure}

\subsection{Virtual knot presentation}\label{subsec:virtual-knot-diag}

In \cite{MR1758871}, Satoh showed how to represent ribbon surfaces of
genus $0$ and $1$ in $\R^4$ by means of virtual knots/links.  For our
purposes, a virtual knot \index{knot!virtual} (or link) is a diagram
in $\R^2$ of embedded, oriented arcs which end either at ``crossings''
as in the (top) first two pictures in Figure~\ref{fig:virtual-knot} or
at endpoints as in the (top) third picture.  Each such diagram
corresponds to a collection of immersed surfaces in $\R^3$ by
replacing each of the crossings and endpoints in
Figure~\ref{fig:virtual-knot} with the corresponding surfaces in
$\R^3$ shown.  These are then connected via tubes parallel to the
embedded arcs.  Thus, any virtual link corresponds to the projection,
with crossing information, of a collection of ribbon $2$-spheres and
tori in $\R^4$.

\begin{figure}[htbp]
  \centering
  \includegraphics{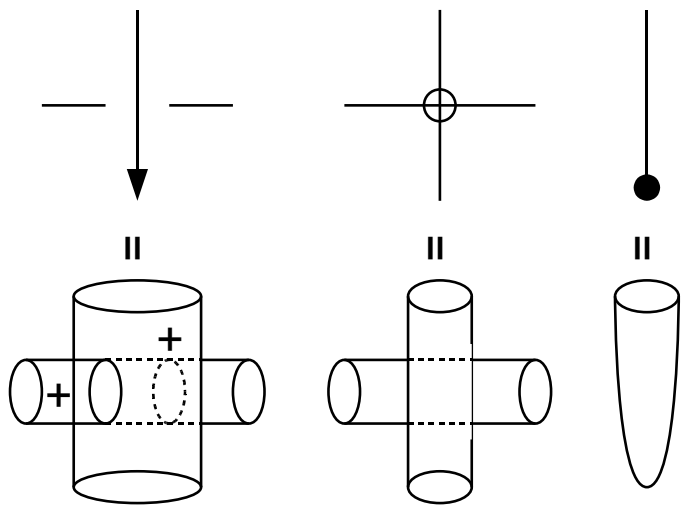}
  \caption[Crossings in virtual knots versus crossings in
  projections]{Correspondence of crossings in virtual knots to
    crossings in projections of surfaces to $\R^3$}
  \label{fig:virtual-knot}
\end{figure} 

Conversely, linear ribbon presentations of knots correspond to virtual
knots.  Take a projection of $K\subset \R^4\to \R^3$ having only
double points at ribbon singularities.  For each band in the linear
ribbon presentation, consider the image of its core in $\R^3$ extended
to the center of the bases to which the band attaches.  This gives an
immersed (at ribbon singularities) arc $\widehat{K}$ in $\R^3$.  Taking a
generic projection $\R^3\to \R^2$ we get an arc $\widecheck{K}$ immersed
in $\R^2$ with two kinds of singularities:
\begin{itemize}
\item double points of the projection $\widehat{K}\subset \R^3 \to \R^2
  \supset \widecheck{K}$ and
\item projections of immersion points $\widehat{K}\into \R^3$.
\end{itemize}
Each of the first kind of double point corresponds to a virtual
crossing.  For the second kind of crossing we must first consider a
diversion about orientations.

The endpoints of $\widecheck{K}$ each correspond to a base with only one
band attached --- here, $K$ locally consists of the discs $D_1,D_2$.
With a fixed orientation on $K$, we orient the boundary of the $D_i$
with the outward normal.  We then say that the endpoint of $\widecheck{K}$
is out/in as the boundary orientation on the $D_i$ is counterclockwise
or clockwise, respectively (when $D_i$ is orientation-preserving
identified with the unit complex disc.) This orients $\widecheck{K}$.

Then, with $\widecheck{K}$ oriented, we can check that the second type of
immersion point corresponds to the ribbon intersection in
Figure~\ref{fig:virtual-knot}.  If it does not (i.e.  the two crossings
have the opposite under/over information) then perform the isotopy in
Figure~\ref{fig:vk-fix}.  Once this has been done, we may use our
correspondence from Figure~\ref{fig:virtual-knot} to label each of the
immersion points of $\widecheck{K}$ as virtual knot crossings.

\begin{figure}[htbp]
  \centering
  \includegraphics{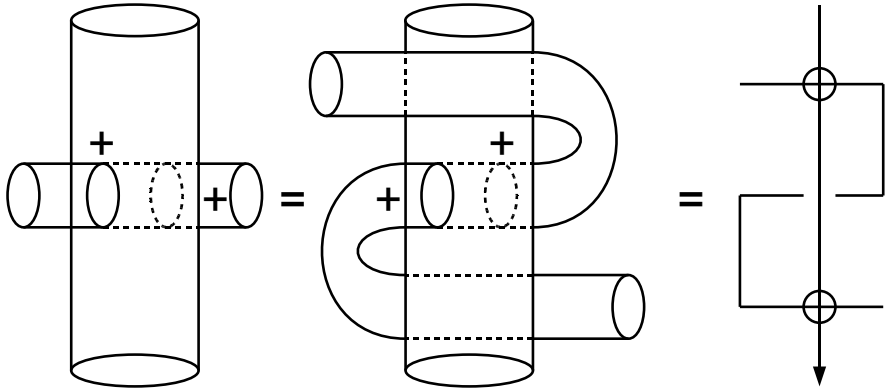}
  \caption[Fixing a ``bad'' ribbon crossing by isotopy]{Fixing a
    ``bad'' ribbon crossing by an isotopy which creates two virtual
    crossings}
  \label{fig:vk-fix}
\end{figure}

In addition to the ``classical'' Reidemeister moves in
Figure~\ref{fig:classical-reidemeister}, associated to a virtual knot,
we have the series of ``virtual'' Reidemeister moves in
Figure~\ref{fig:virt-reidemeister} giving allowable isotopies.  Notice
that move D is one of the forbidden moves of the virtual knots of
Kauffman.  The type of virtual knot we consider here is sometimes
referred to a being {\em weakly} virtual but in the spirit of brevity
we will omit ``weakly'' in this paper.  These concepts of virtual
knots are inequivalent as there are virtual knots, in the sense of
Kauffman, which are knotted and which, when move D is allowed, are
unknotted.  Such an example is given in \cite{MR1758871}.

\begin{figure}[htbp]
  \centering
  \includegraphics{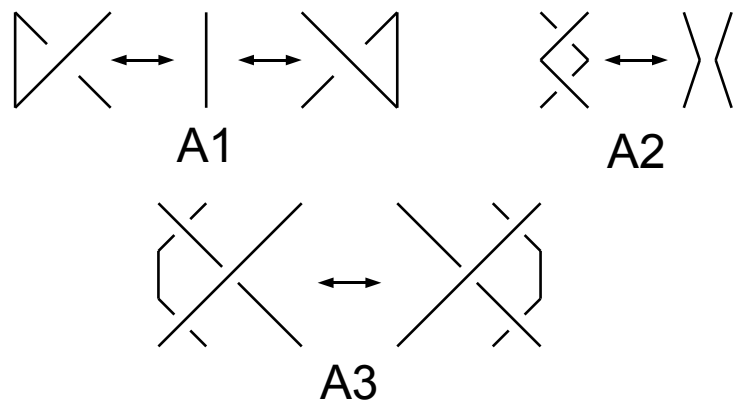}
  \caption{Reidemeister Moves for Classical knots}
  \label{fig:classical-reidemeister}
\end{figure} \index{Reidemeister move!classical}

\begin{figure}[htbp]
  \centering
  \includegraphics{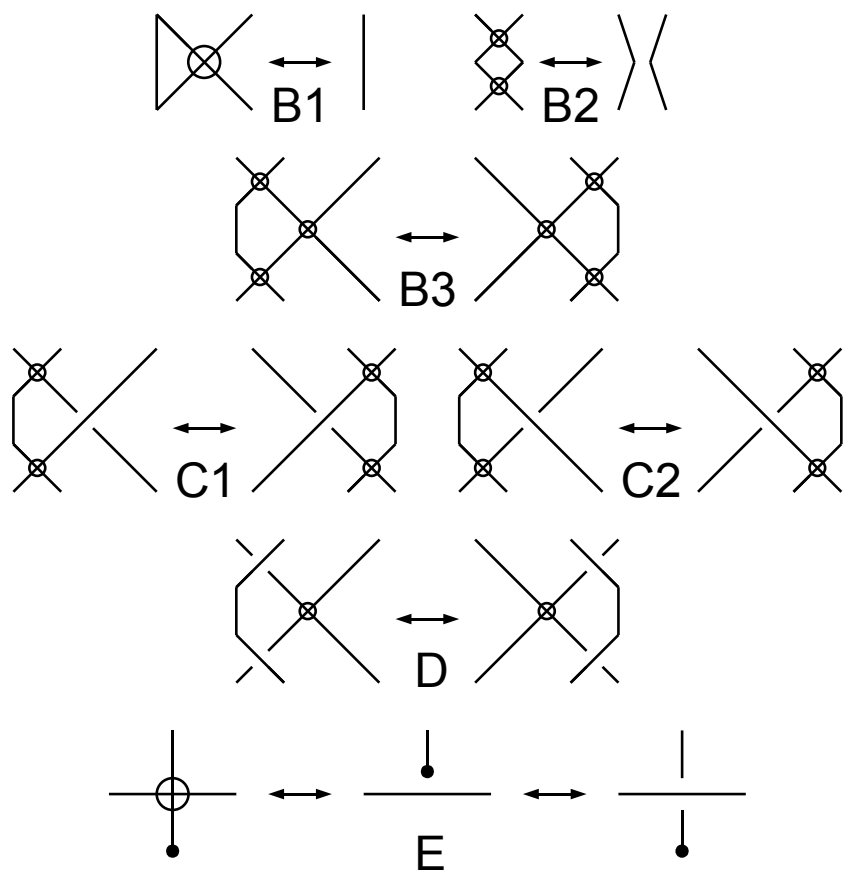}
  \caption{``Reidemeister'' Moves for Virtual knots}
  \label{fig:virt-reidemeister}
\end{figure} \index{Reidemeister move!virtual}

We will add additonal markings to describe ribbon twins.  For a ribbon
twin $\Tw=K_1\union K_2$, there are two bases $D'_i,D''_i$ in the
ribbon representation of each $K_i$ which correspond to the twin
intersection points $K_1\cap K_2$.  Perform band slides until the
ribbon presentations of the $K_i$ are linear with endpoints
$D'_i,D''_i$.  Then we have corresponding virtual knot representations
of the $K_i$ with identical endpoints.  We will use $\oplus,
\ominus$ to mark each of these as they correspond to the cone on the
positive and negative Hopf bands at the intersection points.  With
this in mind, we get the moves in
Figure~\ref{fig:reidemeister-move-f}.

\begin{figure}[htbp]
  \centering
  \includegraphics{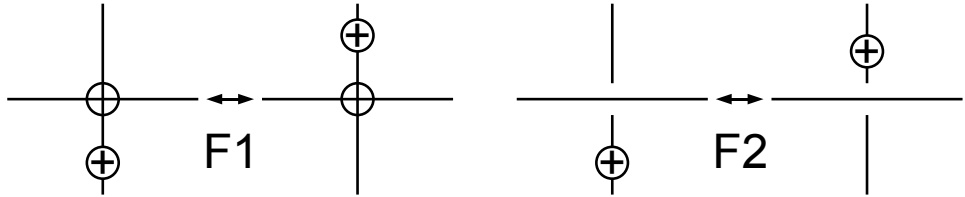}
  \caption[``Reidemeister'' Move F for Twins]{``Reidemeister'' Moves
    $F1$,$F2$ for Twins, versions for $\ominus$ are identical}
  \label{fig:reidemeister-move-f}
\end{figure} \index{Reidemeister move!twin}

Now let us consider the connect sum of a ribbon $2$-knot $K_0$ with
standard twins $\Tw=K_1\union K_2$ as described in
Section~\ref{sec:ribbon}.  Recall that, in this construction, we
formed the connect sum by adding a single band between bases of $K_1$
and $K_0$. We can assume that $K_1$ has a ribbon presentation with $3$
bases, $D_+,D_0,D_-$ with the positive/negative twin intersections
occuring at $D_+$ and $D_-$ respectively. This ribbon presentation has
$2$ bands, $B_+$ from $D_+$ to $D_0$ and $B_-$ from $D_0$ to $D_-$.
Assume that $K_0$ is given by a virtual knot diagram and hence that it
has a linear ribbon decomposition. Let $B_0$ be a band connecting the
base $D_0$ on $K_1$ and endpoints of the linear ribbon decomposition
of $K_0$. Then $K_1\# K_0$ has a ribbon decomposition in the shape of
a $T$, with the top bar cosisting of the bands $B_\pm$ and the shaft
consiting of $B_0$ and the ribbon decomposition for $K_0$.

To get a virtual knot diagram for the twin $\Tw\# K_0$, we will need
to perform a series of band slides -- sliding one of $B_\pm$ along the
ribbon decomposition of $K_0$ to the other endpoint. This is
straightforward and results in a ribbon decomposition whose virtual
knot diagram can be obtained from that of $K_0$ by
\begin{enumerate}
\item replace each strand of the virtual knot for $K_0$ with two parallel strands,
\item at the two endpoints of the virtual knot, replace the endpoint
  with one of diagrams show in the bottom of the first two columns of
  Figure~\ref{fig:conn-sum}, using both,
\item replace crossings for the virtual knot for $K_0$ with
  configurations as shown in Figure~\ref{fig:conn-sum}.
\end{enumerate}
An example is shown in Figure~\ref{fig:giller-twin0}.

\begin{figure}
  \centering
  \includegraphics{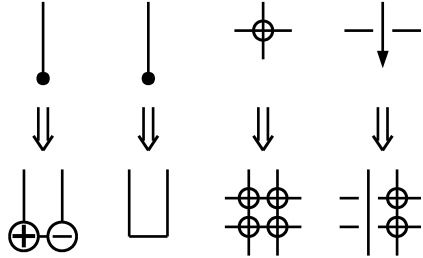}
  \caption{Ribbon knot to connect sum ribbon twin}
  \label{fig:conn-sum}
\end{figure}

\begin{figure}[htbp]
  \centering
  \includegraphics{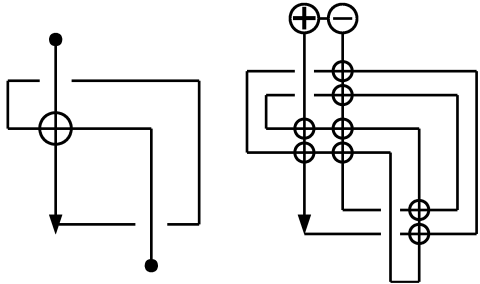}
  \caption{Giller's example (left) and its connect sum with standard twins (right)}
  \label{fig:giller-twin0}
\end{figure}

Another technique for obtaining a Ribbon twin from a ribbon $2$-knot
$K_1$ is to take a virtual knot presentation for $K_1$ and perform
virtual Reidemesiter moves so that the endpoint-bases sit in the
unbounded region of the plane. Then connect these using an
crossingless arc sitting in the unbounded face. This gives the second,
unknotted, $2$-sphere $K_2$ of the twin. For example, see
Figure~\ref{fig:giller-twin}

\begin{figure}[htbp]
  \centering
  \includegraphics{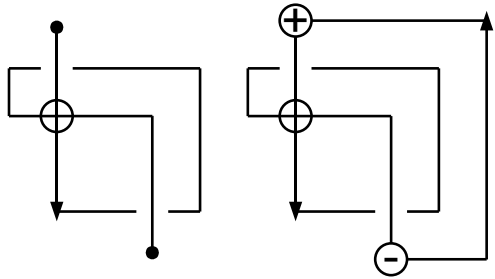}
  \caption[Giller's example and twin version]{Giller's example (left) and twin
    containing it (right)}
  \label{fig:giller-twin}
\end{figure}

\subsection{Surgery diagrams} \label{subsec:surgery-diagram}

As discussed earlier, a twin in $S^4$ has a canonical surgery
associated to it.  Since our decorated projections determine isotopy
type, no additional information is needed to carry out surgery.  For a
$T^2$ in $S^4$, however, we will need additional information.

As any $T^2\subset S^4$ is nullhomologous, it bounds a Seifert
manifold which, via its inward normal, gives a Seifert framing for
the $T^2$.  This gives us a decomposition, $\partial \nu T^2=T^2\times
S^1$.  As surgery replacing $\nu T^2$ with $T^2\times \D^2$ is
determined by the image of $\partial \D^2$, we see that we can
entirely describe surgery by specifying a curve on $T^2$ and an
integer giving the winding about a meridian (boundary of normal disc)
to $T^2$.  See Figure~\ref{fig:surgery-info}.

When the $T^2$ is ribbon with a linear ribbon presentation and
corresponding virtual knot diagram, we can decompose $T^2$ in the
following manner: Let $C$ be the core of the ribbon presentation,
projected to $\R^3$.  Let $\alpha$ be an essential loop on $T^2$ which,
when projected to $\R^3$ is null-homologous in $\R^3\setminus C$.  Any
such loop represents the same homology class on $T^2$.  Let $\beta$ be
$\partial \D^2\times \{t\}$ in a band in the ribbon
presentation.  Orient $\alpha$ to coincide with the orientation of the
virtual knot diagram.  Then orient $\beta$ so that $\alpha\cdot \beta
=+1$ with respect to the orientation on $T^2$.  So $T^2\cong
\alpha\times \beta$.  Then, in the virtual knot diagram, labeling the
knot corresponding to $T^2$ with $(\gamma,\beta/\alpha)$ where
$\gamma$ is the winding number of the attached $\partial \D^2$ with
respect to the Seifert framing and $\beta/\alpha$ is the slope of
$\partial \D^2$ projected to $T^2$.

\begin{figure}[htbp]
  \centering
  \includegraphics{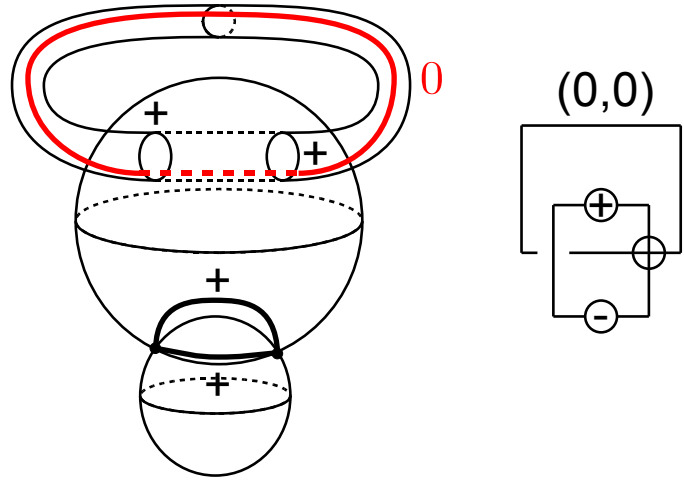}
  \caption{Projection and Virtual Knot surgery diagrams for a twin and
    torus.}
  \label{fig:surgery-info}
\end{figure}

We will write an $S$ together with $*$s on the appropriate components
when we wish to denote this surgery.

\subsection{Giller's Polynomial}\label{sec:giller}

In \cite{MR674227}, C.  Giller defines a polynomial $\Delta_G(t)$ of
simply knotted $S^2$s in $S^4$.  This supposed invariant obeys a skein
relation similar to that of the Alexander polynomial for classical
knots.

That is, consider a embedded circle of double points in a projection
of a (collection of) oriented sphere or torus in $\R^4$.  As mentioned
before, we can trivialize the neighborhood of the double points so
that we have the neighborhood of a classical knot crossing times
$S^1$.  All surfaces in question are oriented and so orient the double
points of their projection --- this orients both strands in the
classical picture.  We can then replace this neighborhood with $S^1$
times any of the $3$ options in Figure~\ref{fig:resolution}, obtaining

\begin{figure}[htbp]
  \centering
  \includegraphics{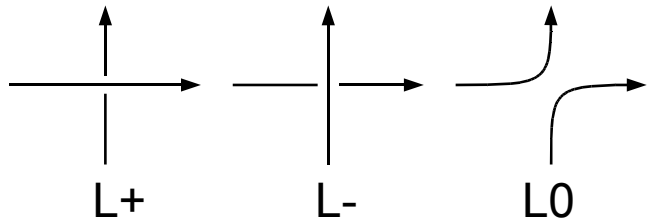}
  \caption{Resolution of a knot crossing}
  \label{fig:resolution}
\end{figure}

The invariant is then defined by the relation:
\begin{equation}
  \label{eq:giller-conway-relation}
  \Delta_G(L_+)-\Delta_G(L_-) = (t^{1/2}-t^{-1/2})\Delta_G(L_0)  
\end{equation}
together with
\begin{equation}
  \label{eq:giller-unknot-relation}
  \Delta_G(\text{unknotted sphere})=1
\end{equation}
and
\begin{equation}
  \label{eq:giller-split-relation}
  \Delta_G(\text{surfaces separated by an }S^3)=0. 
\end{equation}

Giller also describes $\Delta_G$ in a manner similar to that of the
Alexander polynomial. That is, letting $M$ be a Seifert manifold for
$K$, he forms the infinte cyclic cover $X$ of $S^4\setminus K$ and
presents $H_1(X)$ as a $\Q[t,t^{-1}]$ module. Then $\Delta_G$ is
defined by $T=\Q[t,t^{-1}]/(\Delta_G)$ where $T$ is the $\Q[t,t^{-1}]$
torsion part of $H_1(X)$.

Whenever $K$ is ribbon, we can choose $M$ to be a punctured
$nS^1\times S^2$ given by the ribbon presentation. It is easy to
verify by standard arguments that isotopies and band-stabilizations of
the ribbon presentation yeild the same $\Delta_G$. Therefore,
$\Delta_G$ is well-defined for ribbon knots.

For Artin spun knots, Giller's polynomial is the Alexander
polynomial.  In the case that we apply these computations to the
projection of the knotted sphere in an Artin spun twin, Giller's
polynomial is the Seiberg-Witten polynomial.  (as shown in
\cite{MR1650308}).

Interestingly, $2$-knots and twins need not have a symmetric Alexander
polynomial.  Giller's polynomial and the Seiberg-Witten polynomial,
however, are symmetric.  For example, see Figure~\ref{fig:giller-ex},
which is the spun right hand trefoil with crossing changes.

The natural questions to ask are then: Is Giller's polynomial an
invariant of $2$-knots? If so, is it equal to the Seiberg-Witten
polynomial for the corresponding twin? For twins, what is the
relationship between the Alexander polynomial and the Seiberg-Witten
invariant? Our invariant provides suggestive evidence that the second
question, at least in some cases, should be answered in the affirmative.

\section{The $4$-dimensional Conway moves}

\subsection{$3$-dimensional Hoste Move}

The main theorem of Fintushel and Stern in \cite{MR1650308} gives a
way of computing the Seiberg-Witten Invariants of classical-knot
surgered $4$-manifolds in terms of the symmetrized Alexander
polynomial of the knot.  The proof relies on a technique J.  Hoste
developed in \cite{MR743990} which is a method for obtaining Kirby
calculus diagrams for so called ``sewn-up $r$-link exteriors'' in
$S^3$.  (Like Fintushel and Stern, we will only consider the case
where $r$-links are actually knots and links.) We discuss a simplified
but sufficient version of the original move below so to demonstrate
the ideas involved.

\begin{figure}[htbp]
  \centering 
  \includegraphics{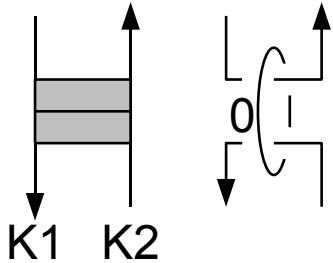}
  \caption{$3$-dimensional Hoste move}
  \label{fig:hoste2}
\end{figure}

A sewn up knot exterior is formed by taking either two oriented knots
in one copy of $S^3$ or in two separate copies, excising a normal
neighborhood of each knot, and gluing the resulting boundary $T^2$s by
a diffeomorphism.  For our purposes, we will let the diffeomorphism be
the one which identifies oriented meridians and longitudes for the
Seifert framings of each knot.

This procedure does two things, it removes two copies of $S^1\times
\D^2$ with a chosen framing and orientation, and it replaces them with
an $S^1 \times S^1 \times I$.  Together, these are the boundary of
$S^1\times \D^2 \times I$.  Thus we may think of forming a sewn up link
exterior as the result (on the boundary) of adding a round
$4$-dimensional $1$-handle to $B^4$ so that the feet of the round
$1$-handle are the two knots, each with the proper framing.

Now, consider a projection of a link $L$ in $S^3$ with oriented
components $K_1$, $K_2$ and a small region in the projection where
$K_1$ and $K_2$ run parallel but in opposite directions.  We can then
connect $K_1$ and $K_2$ via an arc.   See
Figure~\ref{fig:hoste2}.  Attach the round handle as above to form the
sewn up link exterior for $K_1$ and $K_2$.  Note that we can choose the
attaching map of the round handle so that in the corresponding
Morse-Bott function, the points $p_1,p_2$ on $K_1,K_2$ where the arc
touches each knot are both connected to the same point $p$ on the
critical $S^1$ by gradient flow lines.

Take a perfect Morse function on the critical $S^1$ of the round
handle so that the index zero critical point is $p$.  This decomposes
the round $1$-handle into a $1$-handle and $2$-handle corresponding to
the $0$- and $1$-handles of the Morse function on $S^1$.

\begin{figure}[htbp]
  \centering
  \includegraphics{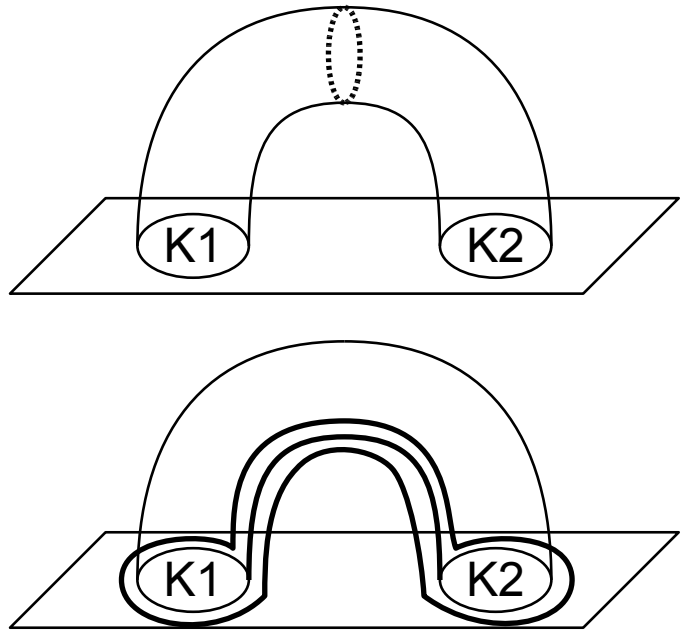}
  \caption{Round handle becomes a $1$ and $2$ handle}
  \label{fig:hoste1}
\end{figure}

Attaching the $1$-handle to $S^3$ results in self connect summing
$S^3$ at the points $p_1,p_2$.  By standard tricks, this is the same as
zero surgery on the unknot around the arc in the second drawing in
Figure~\ref{fig:hoste2}.

The attaching circle of the $2$-handle is the ``band sum'' $K$ of
$K_1$ and $K_2$ as shown in Figure~\ref{fig:hoste1} and the second
picture in Figure~\ref{fig:hoste2}.  Attaching the $2$-handle to the
result (an $S^1\times S^2$) of the previous surgery is then merely
zero surgery on $K$.

\subsection{$4$-dimensional Hoste Move}

In \cite{MR1650308}, the Hoste move shows up in $4$-dimensions with an
$S^1$ equivariance as we cross the $3$-manifold with $S^1$.  When that
is done, the surgeries on knots show up as surgeries on square zero
tori which, by using \cite{MR1492130}, are amenable to computations of
the Seiberg-Witten invariant.  The $4$-dimensional version of the
Hoste move we will discuss here does not assume this $S^1$
equivariance, although local $S^1$ equivariances will occur.

\begin{proposition} \label{prop:4d-hoste1} Consider two embedded,
  oriented square zero tori $T_1,T_2$ in a $4$-manifold $X$.  Suppose
  that $T_1,T_2$ are connected by an annulus $A=S^1\times I$, embedded
  in $X$, so that $\partial A$ consists of an essential curve on each
  torus.  Let each $T_i$ be framed so that $A\cap \partial \nu T_i$ is
  in the subspace of $H_1(\partial \nu T_i)=H_1(T^3)$ generated by the
  pushoffs of loops on $T_i$ with respect to the framing.  Let $\phi$
  be a diffeomorphism $T_1\to T_2$ which identifies the components of
  $\partial A$ in $T_1$ and $T_2$.  Then the self fiber sum,
  $X\fibersum_{T_1=_{\phi}T_2}$, is also the result of surgery on two
  tori: the ``band sum'' of the tori along $A$ and torus given by the
  loop in Figure~\ref{fig:band-sum} in the neighborhood of
  $\theta\times I \subset A$ for each $\theta\in S^1$.
\end{proposition}

\begin{figure}[htbp]
  \centering
  \includegraphics{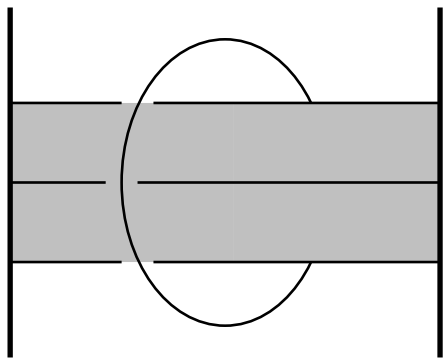}
  \caption{Band sum}
  \label{fig:band-sum}
\end{figure}

\begin{proof}[Proof of Proposition~\ref{prop:4d-hoste1}]
  We can reinterpret the fiber sum as the result (on the boundary) of
  adding a $5$-dimensional toric $1$-handle (a $T^2\times \D^2\times
  I$) to $X\times I$ so that the attaching region, $T^2\times \D^2
  \times \pm 1$ is identified with the normal bundles to $T_1$ and
  $T_2$ with their chosen framing.  This results in deleting the two
  $T^2\times \D^2$s and replacing them with a $T^2\times \partial \D^2
  \times I$.  This identification is determined by choice of framings
  for $T_1,T_2$ and a diffeomorphism $\phi$ between them.

  Choose a factorization $T_1\cong_{\phi} T_2 \cong S^1_\alpha\times
  S^1_\beta$ so that the $S^1_\alpha$ factor is $\partial A$ in both
  $T_i$s.  The critical $T^2$ for the $5$-dimensional toric handle is
  identified with the $T_i$s by the gradient flow.  Pick a perfect
  Morse function on $S^1_\beta$ and perturb the Morse-Bott function on
  the $5$-dimensional handle by an extension of it.  This gives us a
  reinterpretation of the $5$-dimensional toric $1$-handle as two
  round ($S^1$) handles --- a round $2$-handle and a round $1$-handle
  --- corresponding to the critical points of the Morse function on
  $S^1_\beta$.

  Consider the round $1$-handle first.  Such a handle is a $S^1\times
  \D^3 \times I$ so that it attaches along $S^1\times \D^3 \times \pm
  1$.  By construction, the two $S^1\times \D^3$s are neighborhoods of
  the components of $\partial A$, with framing given by the inward
  normal along $A$, a vector field along $\partial A$ parallel to
  $T_i$, and a third vector field defined by orthogonality to these
  and the tangent space to $\partial A$.

  Consider a neighborhood of $A$ which is $S^1$ equivariant, matching
  the $S^1$ equivariance of $A=I\times S^1$.  When small, such a
  neighborhood is diffeomorphic to $S^1$ times the ``H'' in
  Figure~\ref{fig:band-sum}.  (The vertical lines are in $T_i$; the
  horizontal, slices of $A$.) Attaching the round $1$-handle is the
  same as (equivariantly) self-connect summing at the places where the
  vertical lines intersect the horizontal core of the band.  In each
  $3$-manifold slice, this is equivalent to performing zero surgery on
  the loop linking the band in Figure~\ref{fig:band-sum}.  Then, in
  turn, this gives us a square zero torus $L$ and a surgery to perform
  on it within the neighborhood of $A$.

  Now, a $5$-dimensional round $2$-handle is a $S^1\times \D^2 \times
  \D^2$ attached along $S^1\times \D^2 \times \partial \D^2$.  Outside
  of the neighborhood of $A$, the attaching torus is equal to the
  $T_i$.  (two annuli) Inside the neighborhood of $A$, the attaching
  torus $T$ %$T=T_1\fibersum_{A} T_2$
  is given by $S^1$ times the boundary of the band in
  Figure~\ref{fig:band-sum}.  (two more annuli) The framing of this
  torus is given by the framings of the $T_i$ outside the neighborhood
  of $A$ and by the inward normal to the band on (each slice of) the
  inside.  This can be seen by band summing pushoffs of the $T_i$.  (By
  hypothesis, $A$ has zero winding with respect to our framing so this
  can be done by band summing in the $S^1$ trivialization.  ) $T$
  inherits a factorization $S^1_a\times S^1_b$ from the $T_i$ by
  noting that the $S^1_\alpha$ factors of the $T_i$ survive and that
  the $S^1_\beta$ factors are themselves band summed.  Attaching the
  round $2$-handle then performs a surgery on this torus which sends
  $\partial \D^2$ to the $S^1_b$ factor.
\end{proof}

\subsubsection{$4$-Dimensional Hoste Move for a Twin and Torus}
\label{sec:twin-torus}

Let us now examine how Proposition~\ref{prop:4d-hoste1} can be
described in terms of surgery information on projections.  First we
will look at the case when the tori come from surgeries on a twin and
a torus.

Let $X$ be the result of standard surgery on a ribbon twin $\Tw$ and
$(0,0)$ surgery on a ribbon torus $T$ in $S^4$, both specified by a
virtual knot diagram as in sections~\ref{subsec:virtual-knot-diag} and
\ref{subsec:surgery-diagram}.  Note that the cores $T_\Tw,T_T$ of each
surgery inherit preferred framings from their surgery description.

Suppose that $\Tw$ and $T$ have a classical knot crossing in the
virtual knot diagram and hence a ribbon intersection.  Then, in a
projection to $\R^3$, there is a neighborhood as in
Figure~\ref{fig:crossing-change-band}.  Consider the annuli shown in
the figure. Each of these annuli are isotopic. This can be seen by the
fact that on the left of each picture, the horizontal surface
overcrosses the vertical surface so the annulus must lie completely
under the horizontal surface to the left of the ribbon
singularity. Thus we can isotope the annulus freely on the left of the
ribbon singularity. Similarly, on the right the annulus lies
completely above the horizontal surface and so we can isotope it
freely on that side.

Now let $B$ be the particular representation of the annulus
corresponding to the particular orientations of the virtual knot
crossing depicted below it. Now, $B$ connects an equator $\gamma_\Tw$
to one of $K_i$ in $\Tw=K_1\union K_2$ to a essential curve $\gamma_T$
on the torus. Since we have done $(0,0)$ surgery on the torus (in
virtual knot notation), the surgery curve (in the projection notation)
meets $\gamma_T$ once.

\begin{figure}[htbp]
  \centering
  \includegraphics{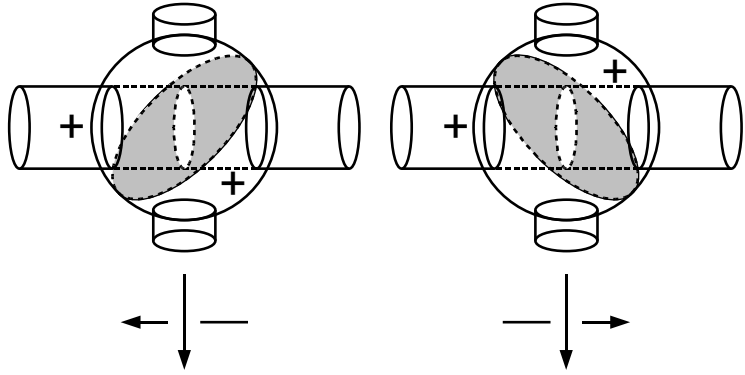}
  \caption{Annulus $B$ used for smoothing}
  \label{fig:crossing-change-band}
\end{figure}

The method of constructing $B$ ensures that $\partial B$ consists of
essential curves on the cores of the twin and torus surgeries when $B$
is extended by the projection to the surgered manifold.  This means
that we can apply proposition~\ref{prop:4d-hoste1} once we have chosen
the diffeomorphism $\phi$ between $T_\Tw,T_T$.  We have already
required that $\phi$ identify the components of $\partial B$.  $\phi$
is then determined when we require that it identify the following:
\begin{enumerate}
\item the projection to $T_T$ of $\text{pt}_1\times\partial \D^2$, where
  $\text{pt}_1\in T$, $T\times \D^2\subset S^4$ the normal bundle, and
\item the projection to $T_\Tw$ of $\text{pt}_2\times\partial \D^2$,
  where $\text{pt}_2\in \Tw\setminus (S^2\text{ intersection points})$
  and lies in the $S^2$ which does not contain the preimage of the
  double points.  $\partial \D^2$ is the boundary of the fiber of the
  normal bundle to this $S^2$.
\end{enumerate}

With this data fixed, proposition~\ref{prop:4d-hoste1} can be applied.
The result is that $$X\setminus(\nu T_{\Tw}\union
\nu T_T))=S^4\setminus(\nu \Tw \union \nu T)$$ sewn up by $\phi$ is
diffeomorphic to $$S^4\setminus(\nu (\Tw \fibersum_{B} T) \union \tau)
\union T^2\times \D^2 \union T^2\times \D^2$$ where
\begin{enumerate}

\item $\Tw\fibersum_{B} T$ denotes the ``band sum'' of $\Tw$ and $T$
  along $B$,

\item the first $T^2\times \D^2$ is glued to $\partial
  \nu( \Tw\fibersum_{B} T)$ by the standard surgery on twins,

\item $\tau$ is the torus given by the loop in
  Figure~\ref{fig:band-sum}, and

\item the second $T^2\times \D^2$ is glued to $\partial \nu \tau$ so
  that $\partial \D^2$ goes to the nullhomologous pushoff of the loop
  in Figure~\ref{fig:band-sum}.

\end{enumerate}

Finally, we can isotopy this region in $\Tw\fibersum_{B} T$ to
be as shown in one of the pictures in the top row of
Figure~\ref{fig:smoothing}.  The appropriate smoothing
depends on the orientation of the horizontal surface and corresponds
to the selection of band previously.  In Figure~\ref{fig:smoothing},
the correspondence of the orientation of the horizontal surface to the
virtual knot diagrams is show by the original diagrams to the lower
left of each smoothing and the smoothed virtual knot diagrams to the
lower right.

\begin{figure}
  \centering
  \includegraphics{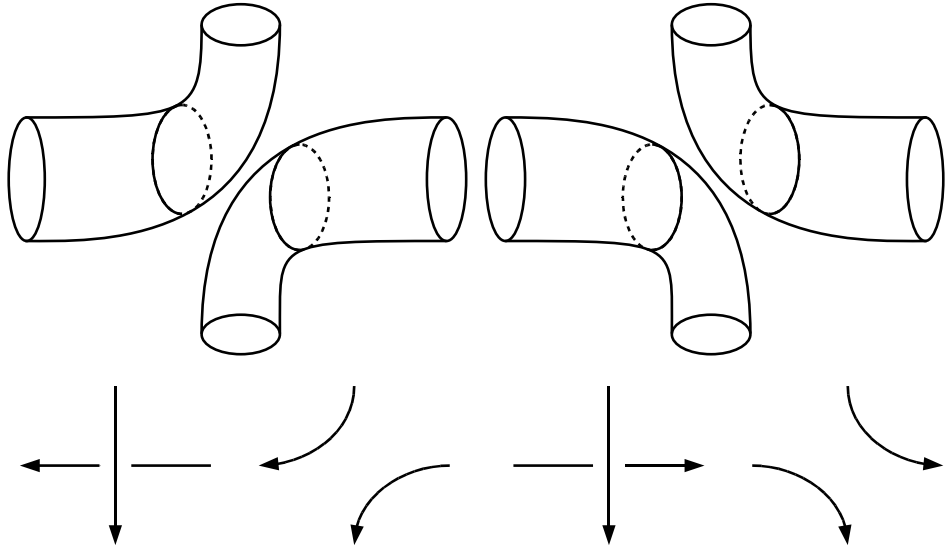}
  \caption[Ribbon intersection, smoothed]{Two smoothings of the
    intersection on Figure~\ref{fig:crossing-change-band} with
    corresponding virtual knot diagrams}
  \label{fig:smoothing}
\end{figure}

\subsubsection{$4$-Dimensional Hoste Move for two Tori}\label{sec:torus-torus}

We will require the Hoste Move between two tori in only one
case. Suppose that both tori lie in a neighborhood diffeomorphic to
$S^1\times \D^3$ so that in each $\theta\times D^3$ the tori are as
shown in Figure~\ref{fig:hoste3}.  The we only need to describe one
aspect of the $3$-dimensional local picture --- the gluing map for the
sewn up exterior for the right hand side of Figure~\ref{fig:hoste3}.
This map is given by identifying meridians to each loop and the
pushoffs along the obvious once punctured discs to each.

\begin{figure}[htbp]
  \centering
  \includegraphics{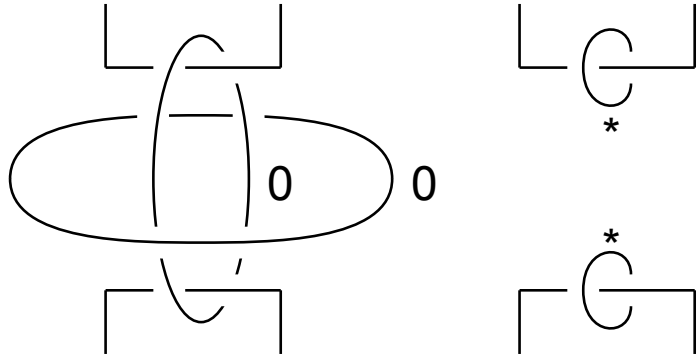}
  \caption[Hoste Move on two tori]{$S^1$ times the left is $S^1$ times
    the right with the two starred tori sewn up}
  \label{fig:hoste3}
\end{figure}

% \subsubsection{$4$-Dimensional Hoste Move for Unconjoined Twins}

% The current lack of a description of a meaningful ``Hoste Move'' for
% unconjoined twins is the major obstruction to defining a complete
% Seiberg-Witten invariant for ribbon twins.  This is an area of
% continuing research for the author.

\subsection{$4$-Dimensional Crossing Change} \label{sec:crossing-change}

Consider a classical crossing in a virtual knot diagram for a twin
and/or a torus and the corresponding annulus from
Figure~\ref{fig:crossing-change-band2}. Notice that the correspondence
is reversed from that of the $4D$ Hoste move. As before, both bands
shown are isotopic. Push the horizontal surface along the annulus in
Figure~\ref{fig:crossing-change-band2} to get the configuration in
Figure~\ref{fig:ribbon-sing-isotope}.

\begin{figure}[htbp]
  \centering
  \includegraphics{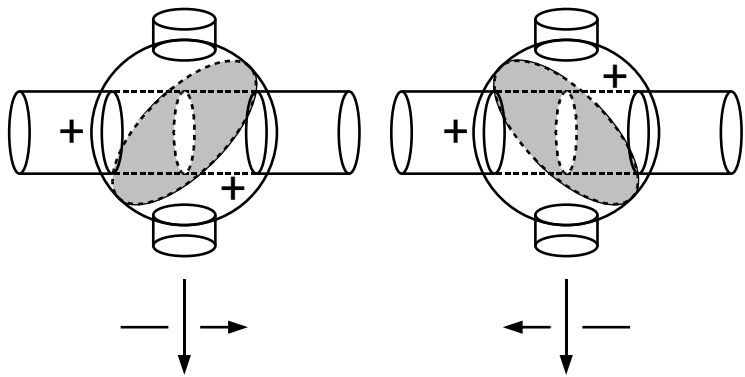}
  \caption{Annulus $B$ used for crossing change}
  \label{fig:crossing-change-band2}
\end{figure}

\begin{figure}[htbp]
  \centering
  \includegraphics{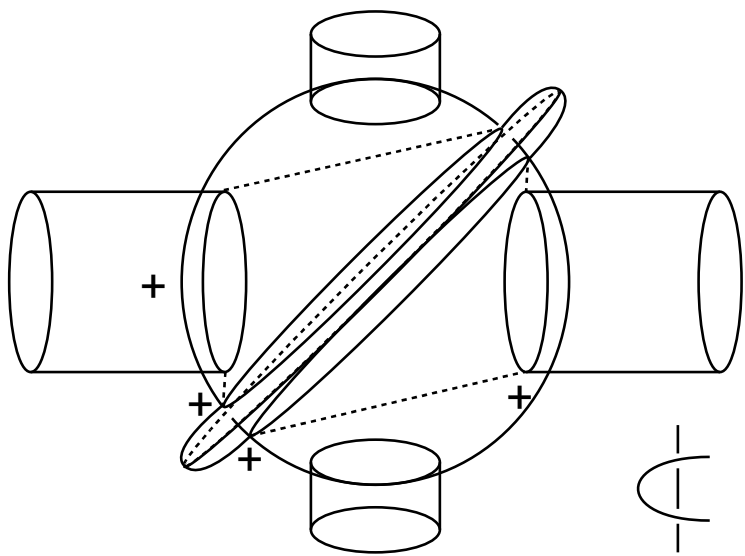}
  \caption[Two new sets of double points]{Diagram obtained from the
    first picture in Figure~\ref{fig:crossing-change-band2} by pushing
    horizontal tube along band, over crossing the vertical tube in the
    two new sets of double points.  The new pair of crossings is
    locally modeled on $S^1$ times the diagram in the lower right.
    The diagram corresponding to the second picture in
    Figure~\ref{fig:crossing-change-band2} is similar.}
  \label{fig:ribbon-sing-isotope}
\end{figure}

Focus our attention to the lower of the two new loops of crossings in
Figure~\ref{fig:ribbon-sing-isotope} (or the corresponding picture for
the other band.) Call this crossing $C$. Local to $C$, we have the
model of $S^1$ times a $3$-dimensional oriented knot crossing. Note
that if we form an annulus by taking a path from the lower to the
upper double point in each $3$-manifold picture, we get an annulus $A$
which is isotopic to the annulus $B$ from before.

Perform one of the surgeries on a torus indicated by the
$3$-dimensional pictures in Figure~\ref{fig:crossing-change-local}
localized at $C$ in Figure~\ref{fig:ribbon-sing-isotope} (or the
corresponding picture for the other band.) The appropriate surgery is
the top for the first picture of
Figure~\ref{fig:crossing-change-band2} and the lower for the
second. This changes the crossing $C$ from an overcrossing of the
horizontal surface to an under crossing, resulting in
Figures~\ref{fig:crossing-change-isotopy}
and~\ref{fig:crossing-change-isotopy2}. Finally, we perform the
isotopies indicated in these figures and see that our result has
changed the crossing type from a classical $+$ to $-$ or from a
classical $-$ to $+$ in the virtual knot diagram.

It is important to note that the torus which we have surgered is
isotopic to the surgered torus $\tau$ of
section~\ref{sec:twin-torus}. We identify the two via the isotopy of
the band $B$.

\begin{figure}[htbp]
  \centering
  \includegraphics{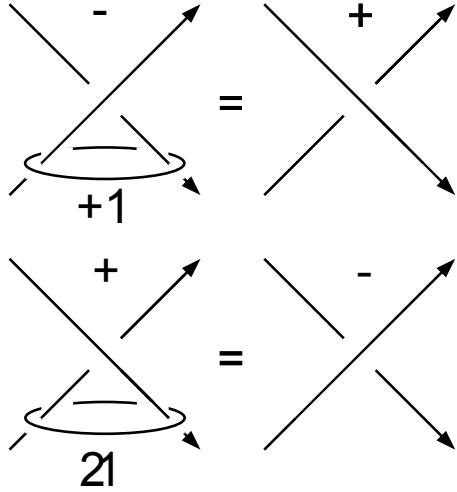}
  \caption{Isotopy in $3$-dimensional picture}
  \label{fig:crossing-change-local}
\end{figure}

\begin{figure}[htbp]
  \centering
  \includegraphics{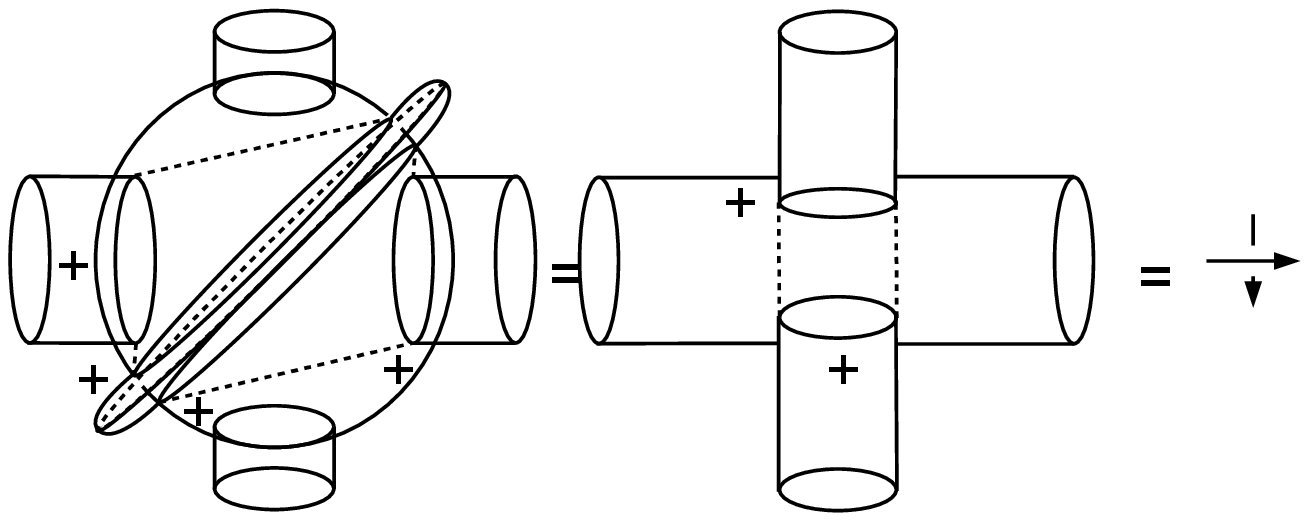}
  \caption[Isotopy of surgered $+$ crossing to $-$ crossing]{Isotopy of
    first picture in Figure~\ref{fig:crossing-change-band2} after
    surgery and the corresponding virtual knot crossing of result}
  \label{fig:crossing-change-isotopy}
\end{figure}

\begin{figure}[htbp]
  \centering
  \includegraphics{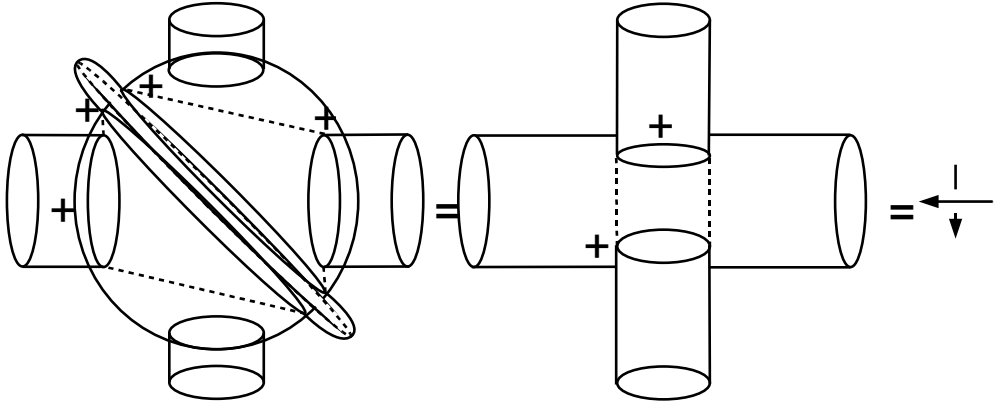}
  \caption[Isotopy of surgered $-$ crossing to $+$ crossing]{Isotopy of
    second picture in Figure~\ref{fig:crossing-change-band2} after
    surgery and the corresponding virtual knot crossing of result}
  \label{fig:crossing-change-isotopy2}
\end{figure}

\section{Calculation of the Invariant for Certain Ribbon Twins}

Consider a twin $\Tw=K_1\cup K_2$ given by a virtual knot presentation
and the manifold $X(\Tw)=E(2)_{\Tw}$.  Suppose that the virtual knot
presentation for $K_1$ contains a classical crossing; so $K_1$ has a
ribbon intersection with itself.  Let $\Tw_+$, $\Tw_-$ and $\Tw_0$ be
the results of replacing the crossing in the virtual knot diagram with
the three options in Figure~\ref{fig:skein-relation}. Note that
$\Tw_0$ will actually be a twin and a torus.

\begin{figure}[htbp]
  \centering
  \includegraphics{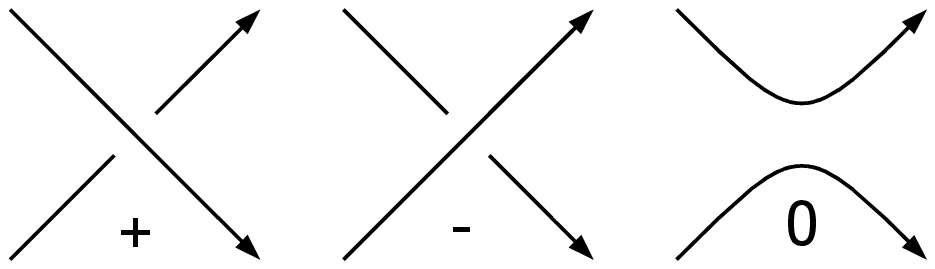}
  \caption{}
  \label{fig:skein-relation}
\end{figure}

Consider the square zero torus $\tau$ from the previous sections.
Now, $\partial \nu \tau$ admits a decomposition $\partial \nu \tau=
S^1_\alpha\times S^1_\beta \times S^1_\gamma$ from it lying in the
$S^1$ equivariant neighborhood of the annulus $B$. Namely, in each
$3$-manifold slice of the neighborhood, $\tau$ is given by a loop
$\alpha$ linking the slice of $B$ once. Then $S^1_\alpha\times
S^1_\beta = \tau$ where $\beta$ is $\text{pt} \times S^1$ in the
equivariant neighborhood of $B$. Finally $S^1_\gamma = \partial \D^2$
finishes the decomposition.

Log transform surgery on $\tau$ is defined by removing $\nu \tau\cong
T^2\times \D^2$ and replacing with another copy of $T^2\times
\D^2$. Such a surgery is uniquely determined up to diffeomorphism by
the image of $\partial \D^2$ in $\partial \nu \tau \cong T^3$. Using
our decomposition above, we can describe such a surgery by a triplet
of integers $(a,b,c)$ with no common factor. Such a triplet gives
specifies the isotopy type of the curve to which we will glue
$\partial D^2$. We will write $X_{\tau(a,b,c)}$ for the result of the
$(a,b,c)$ log transform surgery on $X$.

From section~\ref{sec:crossing-change} we see that we can move from
$\Tw_-$ to $\Tw_+$ by surgery on $\tau\subset S^4\setminus \Tw_-$
which in each $3$-manifold slice is $+1$ surgery on the loop which
gives $\tau$. Thus we can move from $\Tw_-$ to $\Tw_+$ by a $(0,1,1)$
log transform on $\tau$. Note that $(0,0,1)$ surgery on $\tau$ is the
identity. 

Morgan, Mrowka, and Szabo's formula in \cite{MR1492130} then gives
\begin{displaymath}
  SW(E(2)_{\Tw_+}) = SW(E(2)_{\Tw_-}) + SW(E(2)_{\Tw_-,\tau(0,1,0)})
\end{displaymath}
Now by our description in section~\ref{sec:twin-torus},
$E(2)_{\Tw_-,\tau(0,1,0)}$ is the sewn up twin/torus exterior of
$\Tw_0$ fiber summed to $E(2)$. Thus, by our definition of $I$,
\begin{displaymath}
  SW(E(2)_{\Tw_-,\tau(0,1,0)}) = (t-t^{-1})I(\Tw_0),  
\end{displaymath}
while
\begin{displaymath}
  I(\Tw_+)=SW(E(2)_{\Tw_+})
\end{displaymath}
and
\begin{displaymath}
  I(\Tw_-) = SW(E(2)_{\Tw_-}).
\end{displaymath}
Therefore,
\begin{displaymath}
  I(\Tw_+) = I(\Tw_-) + (t-t^{-1}) I(\Tw_0)
\end{displaymath}

Now consider a twin $\Tw=K_1\cup K_2$ and torus $\T$ given by a
virtual knot presentation and the manifold $X(\Tw)=E(2)_{\Tw}$.
Suppose that the virtual knot presentation contains a classical
crossing between $K_1$ and $\T$; so $K_1$ has a ribbon intersection
with $T$.  Let $\Tw_+$, $\Tw_-$ and $\Tw_0$ be the results of
replacing the crossing in the virtual knot diagram with the three
options in Figure~\ref{fig:skein-relation}. Note that $\Tw_\pm$ will
each be a twin and a torus while $\Tw_0$ will be a single twin.

As before, we consider the torus $\tau \subset S^4\setminus \Tw_-$.
From section~\ref{sec:crossing-change} we see that we can move from
$\Tw_-$ to $\Tw_+$ by surgery on $\tau\subset S^4\setminus \Tw_-$
which in each $3$-manifold slice is $+1$ surgery on the loop which
gives $\tau$. Thus we can move from $\Tw_-$ to $\Tw_+$ by a $(0,1,1)$
log transform on $\tau$. Note that $(0,0,1)$ surgery on $\tau$ is the
identity.

Morgan, Mrowka, and Szabo's formula in \cite{MR1492130} then gives
\begin{displaymath}
  SW(E(2)_{\Tw_+}) = SW(E(2)_{\Tw_-}) + SW(E(2)_{\Tw_-,\tau(0,1,0)})
\end{displaymath}
Now, since each $\Tw_\pm$ are composed of a torus and a twin, the
manifolds $E(2)_{\Tw_\pm}$ are each sewn up twin/torus exteriors fiber
summed to $E(2)$. Thus $SW(E(2)_{\Tw_\pm})=(t-t^{-1})I(\Tw_{\pm})$.

Nearby $\tau$ we have a local model for $\Tw_-$ given in the left
picture of Figure~\ref{fig:hoste4}. If we perform the $4D$ Hoste move
on $E(2)_{\Tw_-,\tau(0,1,0)}$ at an annulus equal to $S^1$ times the
horizontal line in this picture, we obtain the picture to the right in
Figure~\ref{fig:hoste4}.

\begin{figure}[htbp]
  \centering
  \includegraphics{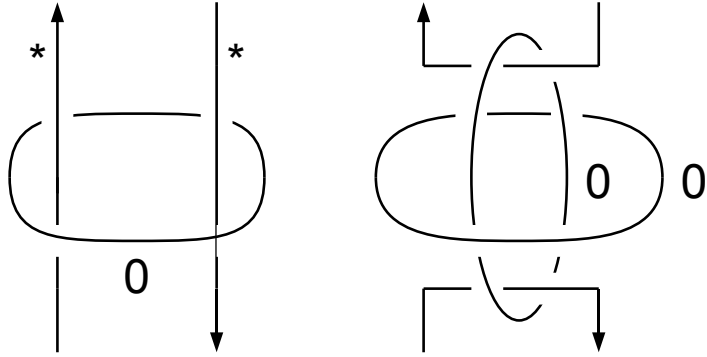}
  \caption{Hoste move nearby $\tau$}
  \label{fig:hoste4}
\end{figure}

Note that this picture is identical to that of
Figure~\ref{fig:hoste3}. Applying the version of the $4D$ Hoste move
from section~\ref{sec:torus-torus}, we get a local picture equal to
that in the right hand side of Figure~\ref{fig:hoste3}. The tori in
this picture are each isotopic to the torus $\partial
\D^2_1\times \partial \D^2_2$ -- where $\partial D^2_i$ is the normal
bundle to $K_i$, the knots which comprise $\Tw_0$. This torus can also
be described as one of the components of $\partial \nu K_1
\cap \partial \nu K_2$. In the fiber sum manifold,
$E(2)_{\Tw_-,\tau(0,1,0)}$, each torus we have just described is
isotopic to the fiber $F$. Thus, $E(2)_{\Tw_-,\tau(0,1,0)}$ is
diffeomorphic to $E(2)_{\Tw_0}\fibersum_{F=F'}$. Therefore,
$SW(E(2)_{\Tw_-,\tau(0,1,0)})=(t-t^{-1})^{2}SW(E(2)_{\Tw_0})$ so
\begin{displaymath}
  SW(E(2)_{\Tw_+}) = SW(E(2)_{\Tw_-}) + (t-t^{-1})^{2}SW(E(2)_{\Tw_0})
\end{displaymath}
and
\begin{displaymath}
  (t-t^{-1})I(\Tw_+) = (t-t^{-1})I(\Tw_-) + (t-t^{-1})^2 I(\Tw_0)
\end{displaymath}
so
\begin{displaymath}
  I(\Tw_+) = I(\Tw_-) + (t-t^{-1}) I(\Tw_0)
\end{displaymath}

\subsection{Relation to Giller's polynomial}

Recall the definition of Giller's polynomial given by Equations
(\ref{eq:giller-conway-relation}), (\ref{eq:giller-unknot-relation}),
and (\ref{eq:giller-split-relation}). These are the Conway-style
relation, the value on the unknotted sphere, and the vanishing of the
polynomial for split links, respectively. We now discuss similar
results for our invariant.

That $I(\Tw_{std})=1$ was shown at
Equation~(\ref{eq:I-std-twin-relation}).

Suppose that $\Tw=K_1\union K_2$ and $T$ are ribbon with a virtual
knot presentation which is split or only has pairwise virtual
crossings. Then using the virtual Reidemeister moves of
Figures~\ref{fig:virt-reidemeister} and~\ref{fig:reidemeister-move-f},
we can separate $\Tw$ and $T$. This means that the projections of
$\Tw$ and $T$ are separated by an $S^3$. So $\Tw$ and $T$ are
separated by an $S^3$. This means that the manifold formed by
surgering $\Tw$ and $T$ is a connected sum. Now, on the $T$ side of
the connected sum, the intersection form on $H_2$ will be a hyperbolic
pair as will the intersection form on the side given by surgery on
$\Tw$. Thus we are given a manifold which is the connected sum of two
manifolds each with $b_+>0$. Therefore, the Seiberg-Witten invariant
of this surgery vanishes. It follows that $SW(E(2)_{\Tw,T})=0$ and
that
\begin{equation}
  \label{eq:I-split-link-relation} 
  I(\Tw,T)=0 \text{ for } \Tw, T \text{ split.}
\end{equation}

Now let us consider the Conway-style relation for Giller's polynomial
we initially discussed in Section~\ref{sec:giller}. This relation
involves crossing changes and resolution at individual loops of double
points. Now, in our crossing change surgery, we had a similar action
of changing the lower crossing from
Figure~\ref{fig:ribbon-sing-isotope}. In the diagrams for our $4D$
Hoste move, we took a different local projection to illustrate the
appropriate surgery. However, smoothing the lower crossing from
Figure~\ref{fig:ribbon-sing-isotope} also yields a smoothing in the
virtual knot diagram.

In other words, selecting a particular set of double points to apply
the relation in Equation~(\ref{eq:giller-conway-relation}) to,
$\Delta_G(\Tw)$ computes $I(\Tw)$. Therefore, $\Delta_G(\Tw)=I(\Tw)$
for ribbon knots. We cannot make a stronger statement of equality
however, as the relation from
Equation~(\ref{eq:giller-conway-relation}) allows us to move into
configurations of surfaces which are inaccessible to the invariant
$I$.

\subsection{The Class of Ribbon Twins}

We now make some remarks on computations.

Now suppose that $\Tw=K_1\union K_2$ is ribbon with a virtual knot
presentation which only has virtual crossings. Then we can use the
virtual Reidemeister moves $B$ and $F$ from
Figures~\ref{fig:virt-reidemeister} and~\ref{fig:reidemeister-move-f}
to completely unknot the diagram for $\Tw$. Therefore, $\Tw=\Tw_{std}$
and so $I(\Tw)=I(\Tw_{std})=1$.

Suppose that $\Tw=K_1\union K_2$ is a twin possibly with accompanying
torus $T$ with the configuration ribbon. Suppose that we reverse the
orientation of one of the $K_i$ or of $T$. This reverses the
orientation of the torus $T_\Tw$ (or $T_T$) and induces a chance in
homology orientation from the change of sign in pairing with
$T_\Tw$. Then, if $\widehat{\Tw}=\overline{K_1}\union K_2$,
$I(\widehat{\Tw})=-I(\Tw)$. Similarly, $I(\Tw,\overline{T})=-I(\Tw,T)$.

Currently, the author is unaware if crossings can be chosen so that
the tree of terminates in standard twins and unlinked twin/torus
pairs. The previous work of Fintushel and Stern guarantees that the
process terminates when $K_1$ in $\Tw=K_1\union K_2$ is knotted with
only classical crossings and $K_2$ is unknotted with no ribbon
intersections with $K_1$. The presence of virtual crossings in the
diagrams complicates the general case. Additionally, the author has
yet to find a general method of dealing with ribbon intersections
between the $K_i$.

However, there seem to be a fairly large number of new examples which
we may compute using the current tools. The first we will compute is
the twin version of the example from Giller's paper. Call this twin
$\Tw_G$. This was encountered previously in
Figure~\ref{fig:giller-twin}.

Follow the computation through
Figures~\ref{fig:computation-giller-twin},
\ref{fig:computation-giller-twin2},
\ref{fig:computation-giller-twin3}, and
\ref{fig:computation-giller-twin4}. In
Figure~\ref{fig:computation-giller-twin4}, we arrive at configurations
$C,D,E,$ and $F$. Here $C$ and $F$ are isotopic to standard twins, so
$I(C)=I(F)=1$. The other configurations $D,E$ differ by the
orientation of the torus and so $I(D)=-I(E)$. Therefore, 
\begin{eqnarray*}
  I(\Tw_G)
  & = & I(A)+(t-t^{-1})I(B) \\
  & = & I(C)+(t-t^{-1})I(D)+(t-t^{-1})I(E)+(t-t^{-1})^2 I(F) \\
  & = & 1 + 0 + (t-t^{-1})^{2} \\
  & = & t^{-2}-1+t^{2}
\end{eqnarray*}

\begin{figure}[htbp]
  \centering
  \includegraphics{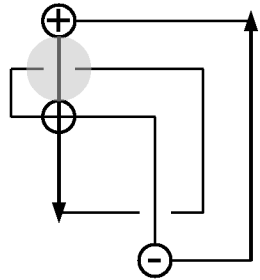}
  \caption{Giller Twin $\Tw_G$ with highlighted crossing}
  \label{fig:computation-giller-twin}
\end{figure}

\begin{figure}[htbp]
  \centering
  \includegraphics{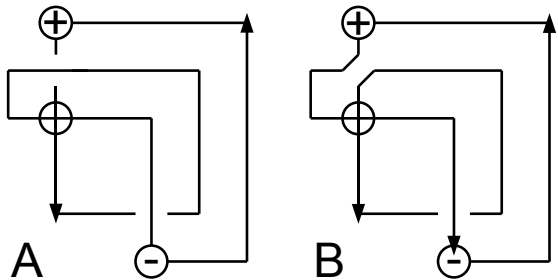}
  \caption{$I(\Tw_G)=I(A)+(t-t^{-1})I(B)$}
  \label{fig:computation-giller-twin2}
\end{figure}

\begin{figure}[htbp]
  \centering
  \includegraphics{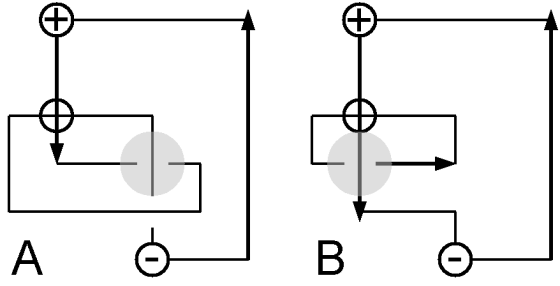}
  \caption{Isotopy of $A$ and $B$ from
    Figure~\ref{fig:computation-giller-twin2} with highlighted crossings}
  \label{fig:computation-giller-twin3}
\end{figure}

\begin{figure}[htbp]
  \centering
  \includegraphics{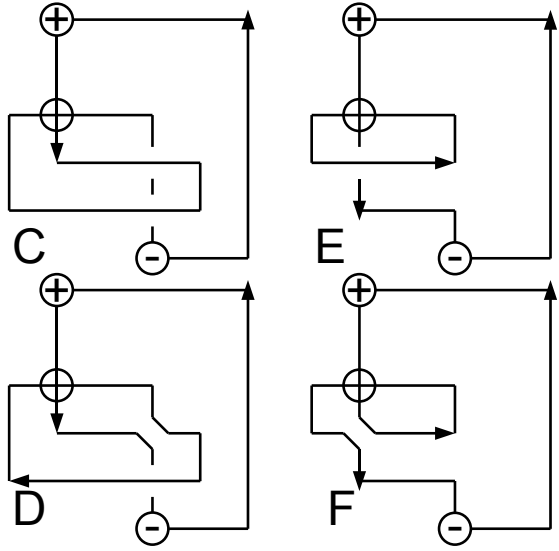}
  \caption{$I(\Tw_G)=I(C)+(t-t^{-1})I(D)+(t-t^{-1})I(E)+(t-t^{-1})^2 I(F)$}
  \label{fig:computation-giller-twin4}
\end{figure}

Now let us look at the twin in Figure~\ref{fig:unknot-pair-twin}, a
twin in which both $2$-knots are unknots but which pairwise have
ribbon intersections. Call this twin $\Tw_U$. Our current tools do not
allow us to deal directly with pairwise ribbon intersections.

Follow the computation through Figures~\ref{fig:unknot-pair-twin1},
\ref{fig:unknot-pair-twin2}, and \ref{fig:unknot-pair-twin3}. We
arrive at configurations $H,K,$ and $L$. Here $H$ and $L$ are isotopic
to standard twins, so $I(H)=I(L)=1$. The other configurations $K$
contains a separated torus so $I(K)=0$. Therefore,
\begin{eqnarray*}
  I(\Tw_U)
  & = & I(H)-(t-t^{-1})I(J) \\
  & = & I(H)-(t-t^{-1})I(K)+(t-t^{-1})^2 I(L) \\
  & = & 1 + 0 + (t-t^{-1})^{2} \\
  & = & t^{-2}-1+t^{2}
\end{eqnarray*}

Finally, we remark on uniqueness and related topics. In what is our
Artin spun case, Fintushel and Stern have conjectured that their knot
surgery construction yields nondiffeomorphic manifolds for
``essentially different'' knots. (Here, ``essentially different''
means that two knots are not isotopic, they are not mirrors, nor they
isotopic under mirroring of connect-summands.) The Alexander
polynomial does not completely distinguish knots however, so the
Seiberg-Witten invariants in their current form do not shed any light
on their conjecture. Similarly, it seems doubtful that the manifolds
$E(2)_{\Tw_G}$ and $E(2)_{\Tw_U}$ are diffeomorphic, but with the
Seiberg-Witten invariants being equal, we have no obvious way in which
to distinguish them. In particular, it seems possible that
$E(2)_{\Tw_U}$ is diffeomorphic to $E(2)_{\Tw_{K}}$ where $\Tw_K$ is
the Artin spin of the left handed trefoil.

Also, results of C. Taubes in \cite{MR1306023} show that a manifold
$X$ with $b_+>2$ admits a symplectic form, the leading term in $SW(X)$
will have coefficient equal to one. The converse to this statement is
known to be false by work of Fintushel and Stern in
\cite{MR1492129}. In the classical (or Artin spun) case, it is
possible to construct a symplectic form on $E(2)_{\Tw}$ when the
classical knot $K$ from which $\Tw$ is constructed is a fibered
knot. While it may be possible to rephrase this construction in terms
of twins, it is unclear what topological conditions are required on
$\Tw$ to achieve the same result. (A sufficient condition is that
$S^4_{\Tw}$ fibers over $T^2$ or $S^2$.)

We then ask, do $E(2)_{\Tw_G},E(2)_{\Tw_U}$ admit symplectic forms?
What conditions on the exterior of the twin guarantee a symplectic
form?

\begin{figure}[htbp]
  \centering 
  \includegraphics{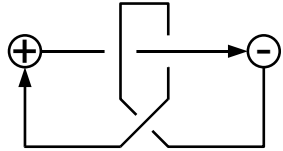}
  \caption{A twin in which both $2$-knots are unknots}
  \label{fig:unknot-pair-twin}
\end{figure}

\begin{figure}[htbp]
  \centering 
  \includegraphics{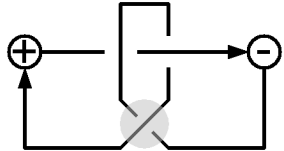}
  \caption{A negative crossing from Figure~\ref{fig:unknot-pair-twin}}
  \label{fig:unknot-pair-twin1}
\end{figure}

\begin{figure}[htbp]
  \centering 
  \includegraphics{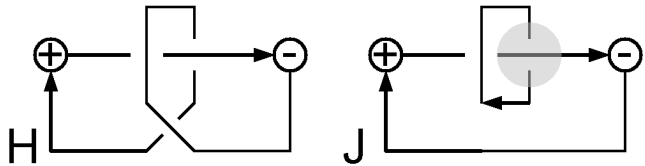}
  \caption{$I(\Tw_U)=I(H)-(t-t^{-1})I(J)$}
  \label{fig:unknot-pair-twin2}
\end{figure}

\begin{figure}[htbp]
  \centering 
  \includegraphics{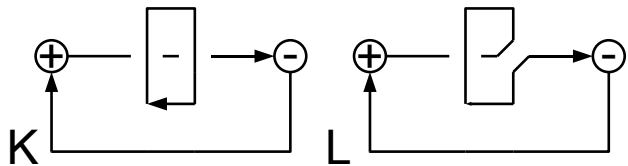}
  \caption{$I(J)=I(K)-(t-t^{-1})I(L)$}
  \label{fig:unknot-pair-twin3}
\end{figure}

\end{document}